\documentclass[12pt, reqno]{amsart}

\usepackage{amssymb, amsmath, amsthm, enumerate, mathtools, graphicx, tikz, color, soul, hyperref, bbm, nameref}
\usepackage{stmaryrd}
\usepackage[margin=1 in]{geometry}
\usepackage[mathscr]{euscript}
\allowdisplaybreaks

\usepackage{parskip}

\usepackage{pgfplots}

\xdefinecolor{darkgreen}{RGB}{0, 153, 0}

\usepackage{amsmath}
\usepackage{verbatim}

\DeclareFontFamily{U}{mathx}{\hyphenchar\font45}
\DeclareFontShape{U}{mathx}{m}{n}{
      <5> <6> <7> <8> <9> <10>
      <10.95> <12> <14.4> <17.28> <20.74> <24.88>
      mathx10
      }{}
\DeclareSymbolFont{mathx}{U}{mathx}{m}{n}
\DeclareFontSubstitution{U}{mathx}{m}{n}
\DeclareMathAccent{\widecheck}{0}{mathx}{"71}
\DeclareMathAccent{\wideparen}{0}{mathx}{"75}

\usepackage{colonequals}

\newcommand{\fH}{\mathcal{H}}

\newcommand{\fD}{\mathcal{D}}

\newcommand{\fG}{\mathcal{G}}
\newcommand{\C}{\mathbb{C}}
\newcommand{\R}{\mathbb{R}}

\newcommand{\Z}{\mathbb{Z}}

\newcommand{\supp}{\operatorname{supp}}

\newcommand{\fC}{\mathcal{C}}
\newcommand{\fQ}{\mathcal{Q}}

\newcommand{\dist}{\operatorname{dist}}

\newcommand{\diam}{\operatorname{diam}}

\newcommand{\im}{\operatorname{im}}
\newcommand{\ch}{\operatorname{ch}}

\newcommand{\graph}{\operatorname{graph}}

\def\XXint#1#2#3{{\setbox0=\hbox{$#1{#2#3}{\int}$}
     \vcenter{\hbox{$#2#3$}}\kern-.5\wd0}}

\newtheorem{theorem}{Theorem}[section]
\newtheorem{proposition}[theorem]{Proposition}
\newtheorem{lemma}[theorem]{Lemma}

\newtheorem{corollary}[theorem]{Corollary}
\theoremstyle{definition}

\newtheorem{remark}[theorem]{Remark}

\numberwithin{equation}{section}
\usepackage{enumerate}

\usepackage{marginnote} 

\begin{document}
{\allowdisplaybreaks 

\renewcommand{\labelenumii}{\arabic{enumi}.\arabic{enumii}}

\title{Two-point patterns determined by curves}
\author{Benjamin B.~Bruce and Malabika Pramanik}
\date{\today}
\thanks{Both authors were supported in part by NSERC Discovery grant GR010263}

\newcommand{\Addresses}{{
  \bigskip
  \footnotesize

{\textsc{Department of Mathematics, 1984 Mathematics Road, University of British Columbia, Vancouver, BC V6T 1Z2, Canada}} \par\nopagebreak
  \textit{E-mail addresses}: \texttt{bbruce@math.ubc.ca, malabika@math.ubc.ca}

}}

\maketitle

\begin{abstract}
Let $\Gamma \subset \R^d$ be a smooth curve containing the origin. Does every Borel subset of $\mathbb R^d$ of sufficiently small codimension enjoy a S\'ark$\ddot{\text{o}}$zy-like property with respect to $\Gamma$, namely, contain two elements differing by a member of $\Gamma \setminus \{0\}$? Kuca, Orponen, and Sahlsten \cite{KOS} answer this question in the affirmative for a specific curve with nonvanishing curvature, the standard parabola $(t, t^2)$ in $\R^2$. In this article, we use the analytic notion of ``functional type'', a generalization of curvature ubiquitous in harmonic analysis, to study containment of patterns in sets of large Hausdorff dimension. Specifically, for {\em{every}} curve $\Gamma \subset \R^d$ of finite type at the origin, we prove the existence of a dimensional threshold $\varepsilon >0$ such that every Borel subset of $\R^d$  of Hausdorff dimension larger than $d - \varepsilon$  contains a pair of points of the form $\{x, x+\gamma\}$ with $\gamma \in \Gamma \setminus \{0\}$.  The threshold $\varepsilon$ we obtain, though not optimal, is shown to be uniform over all curves of a given ``type". We also demonstrate that the finite type hypothesis on $\Gamma$ is necessary, provided $\Gamma$ either is parametrized by polynomials or is the graph of a smooth function.  Our results therefore suggest a correspondence between sets of prescribed Hausdorff dimension and the ``types" of two-point patterns that must be contained therein.
\end{abstract}

\section{Introduction} \label{Introduction}

Hausdorff dimension is a notion of size universally used in geometric measure theory. The aim of this article is to establish Hausdorff dimension as an identifier of two-point ``curved'' patterns contained in a set, with an appropriate interpretation of curvature. 

Let us start with the necessary terminology.
\begin{itemize}
\item{A {\em{smooth curve}} in $\mathbb R^d$ refers to the image of an infinitely differentiable function $\Phi\colon \mathtt I \rightarrow \mathbb R^d$, where $\mathtt I \subset \mathbb R$ is any nondegenerate compact interval.  Here and throughout, we adopt the convention that derivatives at the endpoints of $\mathtt I$ are one-sided.  We also assume, from now on, that all compact intervals $\mathtt I$ are nondegenerate, i.e.~have positive length.  For any smooth curve $\Gamma$, there are many functions $\Phi$ and intervals $\mathtt I$ such that $\Gamma = \Phi(\mathtt I)$.}
\item{A smooth curve $\Gamma \subset \R^d$ will be called \emph{unavoidable} if there exists a constant $\varepsilon > 0$ such that $(\Gamma \setminus \{0\}) \cap (K - K) \ne \emptyset$ for every Borel set $K \subseteq \R^d$ with $\dim_{\text{H}} K > d-\varepsilon$;  if there is no such $\varepsilon$, then $\Gamma$ is \emph{avoidable}.  Here and throughout, $\dim_{\text{H}}$ denotes the Hausdorff dimension; see \cite[Chapter 4, \S 4.3]{Mattila}.}
\end{itemize}

An unavoidable curve necessarily contains the origin.  Indeed, suppose $\Gamma$ is a smooth curve that does not contain the origin.  Since $\Gamma$ is compact, we have $\dist(0, \Gamma) > 0$.  It follows that $\Gamma \cap (K-K) = \emptyset$ for any set $K$ of diameter strictly less than $\dist(0,\Gamma)$, including Euclidean balls which not only have full Hausdorff dimension but positive Lebesgue measure.  Equivalently stated, if $\Gamma$ is an unavoidable curve, then sets of large Hausdorff dimension but arbitrarily small diameter would contain pairs of the form $\{x, x + \gamma\}$ with $\gamma \in \Gamma$, forcing the origin to be a limit point (and thus an element) of $\Gamma$. Unavoidability is therefore a local property of a curve, reflecting its behaviour near the origin.

A classical result of Furstenberg \cite{Furstenberg} and S\'ark$\ddot{\text{o}}$zy \cite{{Sarkozy1}, {Sarkozy2}} shows that the difference set $A - A$ of any set $A \subseteq \mathbb Z$ of positive upper density cannot be square-free. The concept of unavoidability extends this notion to Euclidean spaces of dimension $d \geq 2$:  If $\Gamma \subset \R^d$ is unavoidable, then the difference set $K-K$ of any ``large" set $K \subseteq \R^d$ must contain a nontrivial element of $\Gamma$.  Recent work  of Kuca, Orponen, and Sahlsten \cite{KOS} has shown that the standard parabola $\{(t,t^2) \colon t \in [-1, 1]\}$ in $\R^2$ is unavoidable, leading to a natural question as to which other curves enjoy this property.  This article addresses this question by presenting results  of two types. We show that a smooth curve is unavoidable if it is of ``finite type'' at the origin.  We also consider two distinct subclasses of smooth curves, namely smooth graphs and polynomial curves, and show that within these classes the finite type hypothesis is necessary (as well as sufficient) for unavoidability.

To state our results, let us recall (e.g.~from \cite[Chapter VIII, \S 3.2]{Stein-HA}) the classical notion of ``type" for a smooth function $\Phi \colon \mathtt I \rightarrow \R^d$ at a point $\mathtt{t} \in \mathtt I$.  
\begin{enumerate}[1.]
\item{We say that $\Phi$ is of \emph{finite type} at $\mathtt{t}$ if for every nonzero vector $u \in \R^d$ there exists an integer $n \geq 1$ such that $u \cdot \Phi^{(n)}(\mathtt{t}) \neq 0$. If there is a nonzero vector $u \in \mathbb R^d$ for which no such $n$ exists, then $\Phi$ is of {\em{infinite type}} at $\mathtt{t}$. \label{Phi-finite-type-def}} 
\item{We say that $\Phi$ \emph{is vanishing of finite type} at $\mathtt{t}$ if $\Phi$ is of finite type at $\mathtt{t}$ and $\Phi(\mathtt{t}) = 0$.}
\item{We say that $\Phi$ is of \emph{type $\mathtt N$} at $\mathtt{t}$ if $\mathtt N$ is the smallest integer with the following property:  For every $u \in \R^d \setminus \{0\}$ there exists $n \in \{1,\ldots, \mathtt N\}$ such that $u \cdot \Phi^{(n)}(\mathtt{t}) \neq 0$.

An easy compactness argument shows that $\Phi$ is of finite type at $\mathtt{t}$ (in the sense of definition \ref{Phi-finite-type-def}) if and only if $\Phi$ is of type $\mathtt N$ at $\mathtt{t}$ for some $\mathtt N$.  Since it is always possible to find a nonzero vector $u \in \mathbb R^d$ such that $u\cdot\Phi^{(\ell)}(\mathtt t) = 0$ for every $\ell \in \{1,\ldots, d-1\}$, the smallest possible value of $\mathtt N$ is $d$.   The case $\mathtt N = d = 2$ corresponds to the situation where the image of $\Phi$ has nonzero curvature at $\Phi(\mathtt t)$.  More generally, $\mathtt N = d$ is equivalent to nonvanishing ``torsion":  $\det(\Phi'(\mathtt t), \Phi''(\mathtt t), \ldots, \Phi^{(d)}(\mathtt t)) \neq 0$. \label{Phi-type-N-def}}
\item{We say that $\Phi$ is \emph{vanishing of type $\mathtt N$} at $\mathtt{t}$ if $\Phi$ is of type $\mathtt N$ at $\mathtt{t}$ and $\Phi(\mathtt{t}) = 0$.}
\end{enumerate}
As stated above, if $\mathtt{t}$ is an endpoint of $\mathtt I$, then the derivatives in definitions \ref{Phi-finite-type-def} and \ref{Phi-type-N-def} are to be understood as one-sided.

The notion of type can be transferred from functions to curves.  For simplicity, we formulate the definition at the origin but note that it easily extends to any other point on a curve.  Let $\Gamma \subset \mathbb R^d$ be a smooth curve containing the origin.
\begin{enumerate}[1.]
\setcounter{enumi}{4}
\item{We say that $\Gamma$ is of {\emph{finite type at the origin}} if there exists a compact interval $\mathtt I \subset \R$, a point $\mathtt{t} \in \mathtt I$, and a smooth function $\Phi\colon \mathtt I \rightarrow \mathbb R^d$ that is vanishing of finite type at $\mathtt t$ such that $\Gamma$ contains the image $\Phi(\mathtt I)$;  otherwise $\Gamma$ is of {\em{infinite type}} at the origin. \label{Curve-finite-type} }
\item{We say that $\Gamma$ is of {\emph{type $\mathtt N$ at the origin}} if $\mathtt N$ is the smallest integer with the following property:  There exists a compact interval $\mathtt I \subset \R$, a point $\mathtt{t} \in \mathtt I$, and a smooth function $\Phi\colon \mathtt I  \rightarrow \mathbb R^d$ that is vanishing of type $\mathtt N$ at $\mathtt{t}$ such that $\Gamma$ contains the image $\Phi(\mathtt I)$. }
\label{Curve-type-N}
\end{enumerate}
In definitions \ref{Curve-finite-type} and \ref{Curve-type-N}, we could require that $\mathtt{t} = 0$, since $\Phi$ vanishing of type $\mathtt N$ at $\mathtt{t}$ is equivalent to $\Phi(\cdot - \mathtt{t})$ vanishing of type $\mathtt N$ at $0$.

It is important to note the distinction between the type of a curve $\Gamma$ and the types of the many functions that represent $\Gamma$. For example:
\begin{itemize} 
\item{Suppose that $\eta \colon [0,1] \rightarrow \mathbb R $ is a smooth increasing function that is vanishing of infinite type at $t=0$; say $\eta(t) = \exp(-1/t^2)$. Then the two functions
\begin{align*}
\Phi_0(t) = (\eta(t), \eta(t)^2) \quad\quad \text{ and }  \quad\quad \Phi(t) = (t^2, t^4)
\end{align*}
both represent the standard parabola in a neighbourhood of the origin.  However, $\Phi$ is vanishing of finite type at $t=0$ while $\Phi_0$ is not.  }
 \item{If a curve passes through the origin more than once, then its type is determined by its ``nicest" behaviour there.  Let $\eta\colon [0,1] \rightarrow \mathbb R$ be a smooth function that is identically 0 near $t = 0$ and identically 1 near $t = 1$.  Suppose that $\Gamma = \im \Phi$, where $\Phi\colon [0,1] \rightarrow \mathbb R^2$ is given by 
\begin{align*}
\Phi(t) \colonequals (t(t-1), (t-1)^3 \eta(t)).
\end{align*}
Then $\Phi^{-1}(0) = \{0,1\}$, and $\Phi$ is of infinite type at $t = 0$ and of finite type ($\mathtt N = 3$) at $t = 1$.  In spite of the behaviour of $\Phi$ at $t = 0$, the curve $\Gamma$ is of finite type at the origin according to definition \ref{Curve-finite-type}.}
\end{itemize}

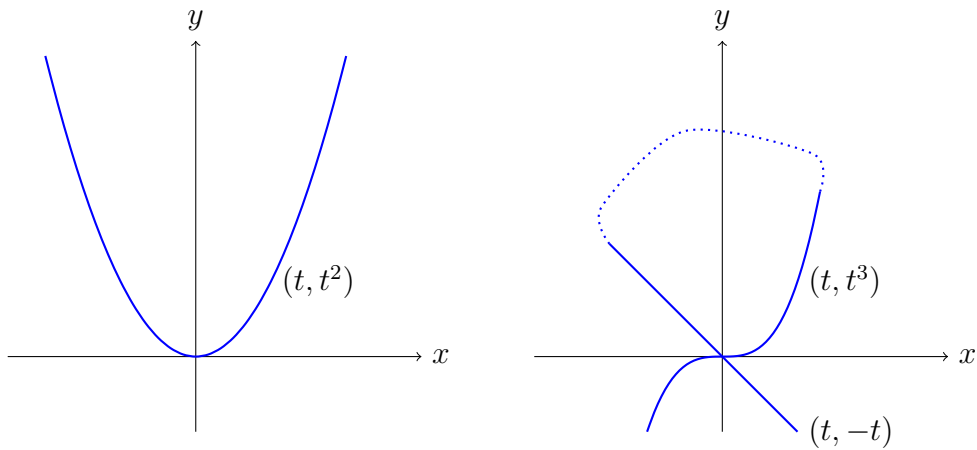
\begin{figure}
\begin{tikzpicture}
\draw[->] (-2.5, 0) -- (3, 0) node[right] {$x$};
\draw[->] (0, -1) -- (0, 4.2) node[above] {$y$};
\draw[scale=1, domain=-2:2, smooth, variable=\x, blue, thick] plot
({\x}, {\x*\x});
\draw (1,1) node[right] {$(t,t^2)$};

\def\Z{7.0}
\draw[->] (-2.5+\Z, 0) -- (3+\Z, 0) node[right] {$x$};
\draw[->] (0+\Z, -1) -- (0+\Z, 4.2) node[above] {$y$};
\draw[scale=1, domain=-1+\Z:1.3+\Z, smooth, variable=\x, blue, thick] plot ({\x}, {(\x-\Z)*(\x-\Z)*(\x-\Z)});
\draw [color=blue, thick]  (-1.5+\Z,1.5) -- (1+\Z,-1);
\draw [dotted, color=blue, thick] plot [smooth] coordinates { (-1.5+\Z,1.5)
(-1.6+\Z, 2) (-.5+\Z,3) (8.2, 2.7) (1.3+\Z, 2.197)};

\draw (1+\Z,1) node[right] {$(t,t^3)$};
\draw (1+\Z,-1) node[right] {$(t,-t)$};

\end{tikzpicture}
\caption{Curves of finite type.}

\end{figure}

\begin{figure}
\begin{tikzpicture}

\def\Z{7.0}
\draw[->] (-2.5, 0) -- (3, 0) node[right] {$x$};
\draw[->] (0, -1) -- (0, 4.2) node[above] {$y$};
\draw[scale=1, domain=-2:2, smooth, variable=\x, blue, thick] plot
({\x}, {(\x)*(\x)*(\x)*(\x)*(\x)*(\x)*0.06});
\draw (1.7,1.7*1.7*1.7*1.7*1.7*1.7*0.06) node[right] {$(t,e^{-1/t^2})$};

\draw[->] (-2.5+\Z, 0) -- (3+\Z, 0) node[right] {$x$};
\draw[->] (0+\Z, -1) -- (0+\Z, 4.2) node[above] {$y$};

\draw [color=blue,thick] (-2+\Z,-1) parabola[bend at end] (-1+\Z,0);
\draw [color=blue,thick] (1+\Z,0) parabola[bend at start] (2+\Z,2);
\draw [color=blue,thick] (-1+\Z,0) -- (1+\Z,0);

\end{tikzpicture}
\caption{Curves of infinite type.}

\end{figure}
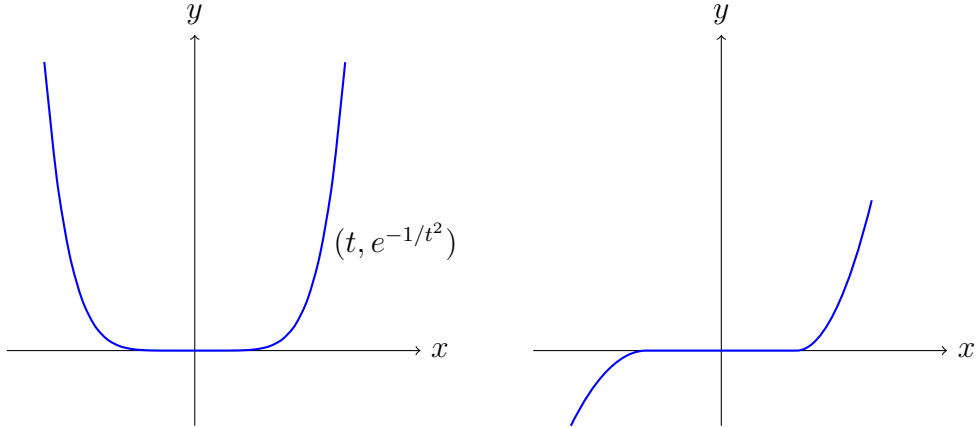

\subsection{Statement of results}  We assume for the entirety of the article that the ambient dimension $d \geq 2$ is fixed.  All constants are allowed to depend on $d$.  The following are our main results.

\begin{theorem}\label{finite type result}
Let $\Gamma \subset \R^d$ be a smooth curve of finite type at the origin.  Then $\Gamma$ is unavoidable.
\end{theorem}
A stronger quantitative version of Theorem \ref{finite type result} is given in Theorem \ref{uniform result} below, where the dimensional threshold is shown to be uniform across the class of curves of fixed type.

As a consequence of Theorem \ref{finite type result}, a full-dimensional set must contain every two-point pattern of finite type.
\begin{corollary}
Let $K \subseteq \R^d$ be a Borel set with $\dim_{\mathrm{H}} K = d$.  Then $(\Gamma \setminus \{0\}) \cap (K - K) \neq \emptyset$ for every smooth curve $\Gamma \subset \R^d$ of finite type at the origin.
\end{corollary}

For certain classes of smooth curves, namely graphs and polynomials, the assumption of finite type at the origin is both sufficient and necessary for unavoidability.
\begin{theorem}\label{graph result}
Let $\Gamma \subset \R^d$ be a curve that contains the origin and is the graph of a smooth function, i.e., of the form
\begin{align}\label{graph form}
\Gamma = \{(t,\Psi(t)) \colon t \in \mathtt I\}
\end{align}
for some compact interval $\mathtt I$ with $0 \in \mathtt I$ and some smooth function $\Psi \colon \mathtt I \rightarrow \R^{d-1}$ with $\Psi(0) = 0$.  Then $\Gamma$ is unavoidable if and only if $\Gamma$ is of finite type at the origin.  
\end{theorem}

\begin{theorem}\label{polynomial result}
Let $\Gamma \subset \R^d$ be a polynomial curve that contains the origin, i.e., $\Gamma = \Phi(\mathtt I)$ for some compact interval $\mathtt I$ and some $d$-tuple $\Phi = (\Phi_1, \ldots, \Phi_d)$ of univariate polynomials $\Phi_i$ such that $\Phi(\mathtt{t}) = 0$ for some $\mathtt{t} \in \mathtt I$.  Then the following are equivalent:
\begin{enumerate}[1.]
\item{$\Gamma$ is unavoidable.}
\item{$\Gamma$ is of finite type at the origin.}
\item{$\Gamma$ is not contained in any hyperplane in $\R^d$.}
\item{There exist linearly independent polynomials $\Phi_1,\ldots, \Phi_d$ and a compact interval $\mathtt I$ such that $\Gamma = \{(\Phi_1(t), \ldots, \Phi_d(t)) \colon t \in \mathtt I\}$.} \label{item2}
\item{If $\Phi_1,\ldots,\Phi_d$ and $\mathtt I$ are any choice of polynomials and a compact interval such that $\Gamma = \{(\Phi_1(t), \ldots, \Phi_d(t)) \colon t \in \mathtt I\}$, then $\Phi_1,\ldots,\Phi_d$ are linearly independent.} \label{item3}
\end{enumerate}
\end{theorem}

We turn our attention now to a quantitative formulation of the qualitative statement in Theorem \ref{finite type result}. For every smooth curve $\Gamma \subset \R^d$ of finite type at the origin, Theorem \ref{finite type result} gives a constant $\varepsilon > 0$, possibly depending on $\Gamma$, such that every Borel set with Hausdorff dimension exceeding $d - \varepsilon$ contains a pair of the form $\{x,x+\gamma\}$ with $\gamma \in \Gamma \setminus \{0\}$. It is natural to ask how this $\varepsilon$ might depend on $\Gamma$. Our proof of Theorem \ref{finite type result} did not provide the optimal threshold value of $\varepsilon$ for a given $\Gamma$; however, a careful scrutiny of the argument revealed that a value of $\varepsilon$ could be chosen so as to depend only on the ambient dimension $d$ and the type of $\Gamma$ at the origin. That $\varepsilon$ therefore works for all curves of the same type as $\Gamma$ at the origin.  We record this as a theorem: 

\begin{theorem}\label{uniform result}
For each $\mathtt N \geq d$, there exists a constant $\varepsilon_{\mathtt N} > 0$ such that $(\Gamma \setminus \{0\}) \cap (K-K) \neq \emptyset$ for every smooth curve $\Gamma \subset \R^d$ of type $\mathtt N$ at the origin and every Borel set $K \subseteq \R^d$  with $\dim_{\mathrm H} K > d-\varepsilon_{\mathtt N}$.
\end{theorem}
These results should be viewed in the context of the substantial body of literature on multi-point patterns in large sets, a genre of problems that has been explored in a variety of discrete and continuous settings; see \cite{{Balog-Pelikan-Pintz-Szemeredi},{Bergelson-Leibman}, {Bloom-Maynard}, {Bourgain-positive-density}, {Bourgain-nonlinear-Roth}, {CLP}, {Denson-Pramanik-Zahl}, {Dong-Li-Sawin}, {Durcik-Guo-Roos}, {Falconer-Yavicoli}, {FGP}, {Gowers}, {GIT1}, {GIT2}, {Han-Lacey-Yang}, {HLP}, {Keleti}, {Keleti2}, {Kuca1}, {Kuca2}, {LP}, {Maga}, {Mathe}, {Peluse1}, {Peluse2}, {Peluse-Prendiville}, {Rice}, {Shmerkin}, {Shmerkin-Suomala}, {Yavicoli1}, {Yavicoli2}}.  It is known that high Hausdorff dimension alone does not guarantee the presence of certain linear patterns, such as three-term arithmetic progressions on the real line or even linear ``parallelograms''; see \cite{Keleti}.  Many pattern existence results (both linear and nonlinear) therefore employ stronger Fourier analytic hypotheses, such as a lower bound on Fourier dimension, or existence of a measure that simultaneously obeys a ball condition and exhibits Fourier decay.  Such results may be found, for example, in \cite{LP, CLP, HLP, FGP}.  In general it is not known whether, for nonlinear patterns, these stronger hypotheses are necessary.  Because of the (heuristic) connection between nonlinearity and Fourier decay, one might hope that Fourier analytic assumptions could be avoided in the nonlinear setting. Theorems \ref{finite type result} and \ref{uniform result} confirm this for two-point patterns determined by curves.

Recent results for more general configurations also align with this view.  Greenleaf, Iosevich, and Taylor  \cite{GIT1, GIT2} have shown that nonlinear patterns are abundant in sets with high Hausdorff dimension. For instance, it is proved in \cite{GIT1} that for each $\Phi \colon \R^d \times \R^d \rightarrow \R^k$ belonging to a large class of smooth maps, there exists a threshold $s_0 > 0$ such that if $K \subseteq \R^d$ has Hausdorff dimension exceeding $s_0$, then the configuration set ${\Delta_\Phi(K) \colonequals \{\Phi(x,y) \colon x,y \in K\}}$ has nonempty interior in $\R^k$.  In particular, their results ensure abundance of patterns determined by curves.  A special case yields the following: If $K \subseteq \R^2$ and $\dim_{\text{H}} K > 3/2$, then the set $\{x_2-y_2-(x_1-y_1)^2 \colon x,y \in K\}$ has nonempty interior; see \cite[Cor.~2.8]{GIT1}.  This result says that sets with high Hausdorff dimension contain two-point patterns determined by a ``continuum'' of parabolas.  Whether any \emph{specific} parabolic pattern must be present is a rather different question.  This was answered by Kuca, Orponen, and Sahlsten \cite{KOS}, as mentioned above:  If $K \subseteq \R^2$ has Hausdorff dimension close enough to 2, then $K$ contains a pair of the form $\{x, x+(t,t^2)\}$ with $t \neq 0$.

The present article shows that many other specific nonlinear two-point patterns can be found in sets with high Hausdorff dimension (Theorems \ref{finite type result} and \ref{uniform result}).  It also offers a characterization of such patterns, provided they are determined by curves corresponding to certain function classes (Theorems \ref{graph result} and \ref{polynomial result}).  These results suggest a correspondence between sets of prescribed Hausdorff dimension and classes of two-point patterns that must be contained therein.  A statement along the lines of the following seems plausible:  

\emph{For every $s_0 \in (0,d)$, there exists a class $\fC = \fC(s_0)$ of smooth curves in $\R^d$ such that (i) $(\Gamma \setminus \{0\}) \cap (K - K) \neq \emptyset$ for every $\Gamma \in \fC$ and every Borel set $K \subseteq \R^d$ with $\dim_{\mathrm H} K > s_0$, and (ii) for every $s < s_0$ and every $\Gamma \in \fC$ there exists a Borel set $K \subseteq \R^d$ with $\dim_{\mathrm H} K > s$ such that $(\Gamma \setminus \{0\}) \cap (K-K) = \emptyset$.}

 We envision that such a class $\fC(s_0)$ might consist of curves of bounded type, for some bound depending quantitatively on $s_0$.

\subsection{Proof structure} The rest of the article is organized as follows:  In Section \ref{sec2}, we prove Theorem \ref{uniform result} and therefore Theorem \ref{finite type result}, conditional on three key propositions that highlight the main technical tools required:
\begin{itemize} 
\item{Given measures $\mu$ and $\pi$, we introduce in Proposition \ref{existence of pattern} a general two-point configuration integral $\mathscr I[\mu; \pi]$ which, if nonzero, signals nonempty intersection of $\supp \mu - \supp \mu$ and $\supp \pi$. }
\item{Given a finite type curve $\Gamma$, the above configuration integral is estimated in Proposition \ref{spectral gap lemma} under assumptions of finite energy and spectral gap for the measure $\mu$, and with $\pi$ being a natural measure supported on $\Gamma \setminus \{0\}$.}
\item{Given a set $K$ of large enough Hausdorff dimension, Proposition \ref{existence of measure}, ensures that a measure $\mu$ satisfying the energy and spectral gap conditions (required for the application of Proposition \ref{spectral gap lemma}) exists on $K$. }
\end{itemize}   
These three propositions are proved in Sections \ref{nonvanishing integral section}, \ref{energy and spectral gap section}, and \ref{measure construction section} respectively, concluding the proof of Theorems \ref{uniform result} and \ref{finite type result}. All three propositions extend statements of a similar nature developed in \cite{KOS} for the standard parabola. However, the current versions are distinctive in the following ways: 

\begin{itemize} 
\item{The proof in \cite{KOS} relies heavily on the anisotropic dilation-invariance of the standard parabola. This feature is not available for general curves, and one of the main contributions of this article lies in providing the necessary workaround. }
\item{Dependencies of the dimensional threshold $\varepsilon$ and the pattern $\gamma \in (K-K) \cap (\Gamma \setminus \{0\})$ on underlying parameters (needed for Theorem \ref{uniform result}) are made explicit. }
\end{itemize} 

In Section \ref{sec3}, we sketch the proof of a known result on H\"older-continuous functions whose graphs have high Hausdorff dimension.  This result plays an essential role in the proofs of Theorems \ref{graph result} and \ref{polynomial result}, which appear in Sections \ref{sec4} and \ref{sec5}, respectively.  In Section \ref{metric space section}, we explain how our methods can be reinterpreted using Hausdorff dimension in a suitable metric space.  A few technical results are relegated to the \nameref{appendix}.

\section{Finite type patterns are unavoidable} \label{sec2}

\subsection{Standardization}
A curve of finite type at the origin may be represented as the image of many functions. The goal of this subsection is to find a parametrization  that is most convenient for later usage.

Let us fix a curve $\Gamma \subset \R^d$ of type $\mathtt N$ at the origin.  Then, by definition \ref{Curve-type-N} on page \pageref{Curve-type-N}, there exists a compact interval $\mathtt I \subset \R$, a point $\mathtt{t} \in \mathtt I$, and a smooth function $\Phi \colon \mathtt I \rightarrow \R^d$ that is vanishing of type $\mathtt N$ at $\mathtt{t}$ such that $\Gamma \supseteq \Phi(\mathtt I)$.  As noted above, we may assume without loss of generality that $\mathtt{t} = 0$, so that $0 \in \mathtt I$.  We may also assume that
\begin{align}\label{domain}
\mathtt I = [0,1].
\end{align}
To justify this, we need to find a function $\widetilde{\Phi} \colon [0,1] \rightarrow \R^d$ that is vanishing of type $\mathtt N$ at the origin such that $\Gamma \supseteq \widetilde{\Phi}([0,1])$.  There are many ways to construct such a function using $\Phi$; one way is the following:  Let $\mathtt I =: [a,b]$, so that $a \leq 0 \leq b$ with at least one of the inequalities being strict, and define
\begin{align*}
\widetilde{\Phi}(t) \colonequals \begin{cases}
\Phi(at) & \text{if } a < 0,\\
\Phi(bt) & \text{if } a = 0
\end{cases} \quad\quad\text{for } t \in [0,1].
\end{align*}
Replacing $\Phi$ by $\widetilde{\Phi}$, we have \eqref{domain}.

For each $i \in \{1, \ldots, d\}$, let $\Phi_i \colon \mathtt I \rightarrow \R$ denote the $i^{\text{th}}$ component of $\Phi$, so that $\Phi = (\Phi_1, \ldots, \Phi_d)$.  The functions $\Phi_i$ can be expressed as
\begin{align}\label{standard form}
\Phi_i(t) = t^{\mathtt{n}_i}\phi_i(t)
\end{align}
for some positive integers $\mathtt n_1,\ldots, \mathtt n_d \in \{1, \ldots, \mathtt N\}$ and some smooth functions $\phi_i \colon \mathtt I \rightarrow \R$ satisfying $\phi_i(0) \neq 0$.  We will make the following assumptions about $\mathtt n_i$ and $\phi_i$: 
\begin{align}\label{ordering}
&1 \leq \mathtt n_1 < \mathtt n_2 <\cdots < \mathtt n_d = \mathtt N, \\  
\label{coefficients}
&\phi_1(0) = \phi_2(0) = \cdots = \phi_d(0) = 1. 
\end{align}
These assumptions also require justification, which is provided in part by Lemma \ref{linear transformation} below.  There, we prove the existence of an invertible linear map ${\bf L} \colon \R^d \rightarrow \R^d$ such that the composition ${\bf L} \circ \Phi$ has the properties \eqref{ordering} and \eqref{coefficients}.  It is straightforward to check that ${\bf L}(\Gamma)$ is unavoidable with a dimensional threshold of $\varepsilon$ if and only if  $\Gamma$ is unavoidable with the same threshold. Therefore,  replacing $\Phi$ by ${\bf L} \circ \Phi$, we may assume without loss of generality that $\Phi$ satisfies the desired properties \eqref{ordering} and \eqref{coefficients}.  The proof of Lemma \ref{linear transformation} is deferred to the \nameref{appendix}.

\begin{lemma}\label{linear transformation}
Let $\Theta \colon \mathtt I \rightarrow \R^d$ be a smooth function that is vanishing of type $\mathtt N$ at the origin.  Then there exists an invertible linear map ${\bf L} \colon \R^d \rightarrow \R^d$ such that $\Phi \colonequals {\bf L} \circ \Theta$ obeys \eqref{standard form} with the accompanying integers $\mathtt n_i$ and functions $\phi_i$ obeying \eqref{ordering} and \eqref{coefficients}, respectively.
\end{lemma}

If $\Phi\colon \mathtt I \rightarrow \mathbb R^d$ obeys \eqref{domain}--\eqref{coefficients}, then we say that $\Phi$ is in {\em{standard form}}. Lemma \ref{linear transformation} (and the discussion preceding it) implies that any curve of finite type at the origin is, after a harmless invertible linear transformation, the image of a function in standard form.

A constant appearing in the proof of Theorem \ref{uniform result} (or any propositions used therein)  is said to be {\em{admissible}} if it depends only on $d$ and $\mathtt N$. In particular, the dimensional threshold $\varepsilon_{\mathtt N}$ provided by the theorem is to be admissible. It will therefore be important to indicate admissibility, or otherwise, of constants that appear in the proof.
 
\subsection{A few key propositions} We now formulate the main steps from which Theorem \ref{uniform result}, and hence Theorem \ref{finite type result}, will easily follow.  Each step, including any accompanying definitions, is a quantitative adaptation of a similar idea occurring in \cite{KOS}. These propositions will be proved in Sections \ref{nonvanishing integral section}, \ref{energy and spectral gap section}, and \ref{measure construction section} respectively.

\subsubsection{nonvanishing of a configuration integral} 
A recurring feature in the study of configurations is the formulation of an appropriate integral whose positivity signals the presence of the desired configuration. We describe below the integral relevant to our problem.

\begin{proposition}\label{existence of pattern}
Fix a Schwartz function $\psi \colon \R^d \rightarrow \C$, and set $\psi_\delta \colonequals \delta^{-d}\psi(\delta^{-1}\cdot)$ for $\delta > 0$.  Let $\mu$ be any compactly supported Borel probability measure on $\R^d$, and let $\pi$ be any finite Borel measure on $\R^d$.  Assume that
\begin{align}\label{configuration integral}
\mathscr{I}[\mu, \pi] \colonequals  \liminf_{\delta \searrow 0}\Big\vert\int (\mu \ast \psi_\delta) \ast \pi \,d\mu\Big\vert > 0.
\end{align}
Then $\supp \pi \cap (\supp \mu - \supp \mu) \neq \emptyset$.
\end{proposition}

In our proof of Theorem \ref{uniform result}, we will apply Proposition \ref{existence of pattern} to measures $\mu$ and $\pi$ supported on (affine images of) $K$ and $\Gamma \setminus \{0\}$, respectively.  This will eventually yield the desired conclusion that $(\Gamma \setminus \{0\}) \cap (K-K) \neq \emptyset$.

Although Proposition \ref{existence of pattern} holds for any choice of Schwartz function $\psi$, we will fix a convenient $\psi$ with properties that simplify certain parts of the subsequent argument.  Specifically, we now take $\psi$ to be the normalized Gaussian
\begin{align}\label{def psi}
\psi(x) \colonequals e^{-\pi |x|^2}
\end{align}
(which is its own Fourier transform) and record the following properties in particular: 
\begin{equation} \label{conditions-psi} 
\psi \geq 0, \quad\quad  \psi(0) = \|\psi\|_\infty = \|\widehat{\psi}\|_\infty = 1, \quad\quad  |\widehat{\psi}(\xi) - \widehat{\psi}(0)| \leq \pi|\xi|^2. 
\end{equation}   
Here, $\widehat{\psi}$ denotes the Fourier transform of $\psi$. For any Borel measure $\mu$, let us define 
\begin{equation} \label{def-mu-psi-delta}
\mu_\delta \colonequals \mu \ast \psi_\delta,
\end{equation}
where $\psi_\delta \colonequals \delta^{-d}\psi(\delta^{-1}\cdot)$, as above.

\subsubsection{Role of energy and spectral gap in identifying patterns} \label{spectral gap proposition section}
Proposition \ref{existence of pattern} is quite general, in the sense that it does not require any special properties of $\mu$ or $\pi$. Our next goal is to ensure that, for a given curve $\Gamma$ of finite type at the origin, \eqref{configuration integral} holds for an appropriate choice of $\mu$ and $\pi$ with $\pi$ (essentially) supported on $\Gamma \setminus \{0\}$. We will also need to describe the admissible and inadmissible constants involved in this choice.

Toward this end, let $\Phi \colon [0,1] \rightarrow \R^d$ be any function in standard form that is vanishing of type $\mathtt N$ at the origin.  For each $i \in \{1,\ldots, d\}$, assumptions  \eqref{standard form}--\eqref{coefficients} imply that
\begin{align*} 
&\Phi_i^{(\mathtt n_i)}(0) = \mathtt n_i! \phi_i(0) = \mathtt n_i! \in [1, \mathtt N!]\\
\intertext{and}
&\lim_{t \searrow 0} \frac{\Phi_i^{(\ell)}(t)}{t^{\mathtt n_i-\ell}} = \frac{\mathtt n_i!}{(\mathtt n_i-\ell)!} \in [1, \mathtt N!] \quad \text{for } 0 \leq \ell < \mathtt n_i.
\end{align*}  
Here, $\Phi_i^{(\ell)}$ denotes the $\ell^{\text{th}}$ derivative of $\Phi_i$, with the convention that $\Phi_i^{(0)} \equiv \Phi_i$. Using the smoothness of $\Phi$ near the origin, one can find a large integer $\mathtt J_0 = \mathtt{J}_0(\Phi)$ depending on $\Phi$ (and therefore inadmissible) such that
\begin{align} \label{c1}
&| \Phi_i^{(\ell)}(t) | \leq 2\mathtt N! |t|^{\mathtt n_i-\ell} \quad  \text{for all } t \in [0, 2^{- \mathtt J_0}] \text{ and } 0\leq \ell < \mathtt n_i,\\
\intertext{and}
&\frac{1}{2} \leq | \Phi_i^{(\mathtt n_i)}(t) | \leq 2 \mathtt N! \quad  \text{for all } t \in [0, 2^{-\mathtt J_0}].  \label{c2}
\end{align}
Property \eqref{standard form} and the lower bound in \eqref{c2} together imply that
\begin{align} 
\Phi(t) \neq 0 \quad \text{for all } t \in (0,2^{-\mathtt J_0}].  \label{nonzero property}
\end{align}  
Thus, the origin is the single isolated zero of $\Phi$ on $[0, 2^{-\mathtt J_0}]$.

Let us define the rescaled functions
\begin{align}\label{rescaled Phi}
\Phi^j \colonequals (2^{\mathtt n_1j}\Phi_1(2^{-j}\cdot),\ldots,2^{\mathtt n_dj}\Phi_d(2^{-j}\cdot)), \quad \quad j \geq 0,
\end{align}
which interpolate between $\Phi$ (when $j=0$) and the monomial curve $t \mapsto (t^{\mathtt n_1},\ldots,t^{\mathtt n_d})$ (when $j = \infty$).  For each $j \geq \mathtt{J}_0$ and $c \in (0,1]$, let $\pi =  \pi[\Phi; j,c]$ denote the singular measure defined by the formula
\begin{align}\label{measure definition}
\int fd\pi  \colonequals \int_{c}^1 f(\Phi^j(s))ds
\end{align}
and supported on $\Phi^j([c,1]) \subset \R^d \setminus \{0\}$.

For $\sigma \in (0,d)$, let $I_{\sigma}(\mu)$ denote the $\sigma$-dimensional energy of a Borel measure $\mu$:
\begin{align}\label{def energy}
I_\sigma(\mu) \colonequals \iint |x-y|^{-\sigma}d\mu(x)d\mu(y) = \gamma(d, \sigma) \int|\widehat{\mu}(\xi)|^2|\xi|^{\sigma-d} d\xi;
\end{align}
here, $\gamma(d, \sigma)$ is a positive constant (see \cite[\S 3.4--3.5]{Mattila2}), and $\widehat{\mu}$ denotes the Fourier transform of $\mu$. For each positive integer $N$, let
\begin{align}\label{def sigma_N}
\sigma_N \colonequals d-\frac{1}{2N} \quad\quad \text{and} \quad\quad \gamma_N \colonequals \gamma(d,\sigma_N).
\end{align}

\begin{proposition}\label{spectral gap lemma}
For each $\mathtt N \geq d$, there exists an admissible constant $\mathtt{L}_{\mathtt N} \geq 1$ with the following property:  Let $\Phi \colon [0,1] \rightarrow \R^d$ be any function in standard form that is vanishing of type $\mathtt N$ at the origin.  Then there exists an (inadmissible) integer $\mathtt J = \mathtt J(\Phi) \geq \mathtt J_0(\Phi)$ such that if
\begin{itemize}
\item{{$\mathtt{A},\mathtt{B},\mathtt{C}$ are any choice of constants satisfying
\begin{equation} \label{AB conditions} 
\mathtt{A}^d \geq 4\mathtt L_{\mathtt N}^2, \quad\quad \mathtt B \geq (\mathtt L_{\mathtt N} \mathtt A^{4d}  \mathtt C)^{2\mathtt N}, \quad\quad \mathtt C \geq 1,
\end{equation}
and}}
\item{$\mu$ is any Borel probability measure on $[0,1]^d$ that obeys the energy and spectral gap conditions
\begin{align}\label{spectral gap condition}
I_{\sigma_{\mathtt N}}(\mu) \leq \mathtt C \quad\quad \text{and} \quad\quad \int_{|\xi| \in [\mathtt A, \mathtt B]}|\widehat{\mu}(\xi)|^2 d\xi \leq \mathtt A^{-4d},
\end{align}}
\end{itemize}
then 
\begin{align}\label{integral lower bound}
\mathscr{I}[\mu, \pi[\Phi; j, \mathtt A^{-6d}]] \geq \mathtt{A}^{-4d} \quad \text{for all $j \geq \mathtt J$}.
\end{align}
\end{proposition}
Here, $\mathtt{J}_0$ refers to the integer appearing in \eqref{c1}--\eqref{nonzero property}, $\sigma_{\mathtt N}$ is the index defined in \eqref{def sigma_N}, $\mathscr{I}$ refers to the configuration integral in \eqref{configuration integral}, and $ \pi$ is the measure introduced in \eqref{measure definition}.

There is nothing special about the choice of $\sigma_{\mathtt N}$ in \eqref{def sigma_N}, other than it being sufficiently close to $d$.  We could have chosen any value for $\sigma_{\mathtt N}$ from the interval $(d-1/\mathtt N, d)$ and Proposition \ref{spectral gap lemma} would still hold.  Similarly, there was some flexibility when choosing the exponent $4d$ that appears in the spectral gap condition in \eqref{spectral gap condition}.  We could replace $4d$ by any number strictly larger than $d+2$, provided we also make minor adjustments to the conditions \eqref{AB conditions} and the conclusion \eqref{integral lower bound}.  The choice of $4d$ simply gives nicer-looking exponents throughout the proof.

\subsubsection{Construction of a measure with finite energy and spectral gap}
Given a finite type curve $\Gamma$ represented by the function $\Phi$, the conclusion \eqref{integral lower bound} and Proposition \ref{existence of pattern} imply, roughly, that $(\Gamma \setminus \{0\}) \cap (\supp \mu - \supp \mu) \neq \emptyset$.  Given an arbitrary Borel set $K \subseteq \mathbb R^d$ of large Hausdorff dimension, it remains to ascertain whether a probability measure $\mu$ obeying the hypotheses of Proposition \ref{spectral gap lemma} can be found with support in $K$.  This is the objective of the next proposition, which says, in short, that such a measure can be found, not in $K$ itself but in a certain affine copy of $K$.

Before stating the proposition, we introduce the necessary notation and terminology.  Fix a $d$-tuple of positive integers $\vec{n} = (n_1,\ldots,n_d)$.  Unlike in \eqref{ordering}, the entries of $\vec{n}$ need not be distinct or ordered. Define
\begin{align}\label{def of D^*}
\fD^{\ast} =  \fD^{\ast}[\vec{n}] \colonequals \bigcup_{j \in \Z} \fD_j, \quad \text{where} \quad \fD_j = \fD_j[\vec{n}] \colonequals \Big\{a + \prod_{i=1}^d [0,2^{-n_i j}) \colon a \in \prod_{i=1}^d 2^{-n_i j}\Z  \Big\}.
\end{align}
Thus, $\mathcal D^{\ast}$ consists of all dyadic boxes in $\mathbb R^d$ of dimensions $2^{-n_1j} \times \cdots \times 2^{-n_d j}$ for some integer $j$.  We also set
\begin{align}\label{def of D_J^*}
\fD_J^{\ast} =  \fD_J^{\ast}[\vec{n}] \colonequals \bigcup_{j \geq J} \fD_j
\end{align}
for any integer $J$.  Now, fix a box $Q \in \fD_j$ and write $Q = a + \prod_{i=1}^d [0,2^{-n_i j})$.  We denote the ``length" of $Q$ by
\begin{align}\label{def ell_Q}
\ell(Q) \colonequals 2^{-j}
\end{align}
and define the rescaling function 
\begin{align}\label{rescaling map}
{\bf T}_Q(x) \colonequals (2^{n_1j}(x_1-a_1), \ldots, 2^{n_dj}(x_d-a_d))
\end{align}
that maps $Q$ onto $[0,1)^d$.  Given a finite Borel measure $\nu$ such that $\nu(Q) > 0$, the \emph{blow-up} of $\nu$ with respect to $Q$ is defined as
\begin{align}\label{def blow-up}
\nu^Q \colonequals \|\nu |_Q\|^{-1}{\bf T}_Q(\nu |_Q), 
\end{align}
where ${\bf T}_Q(\nu |_Q)$ denotes the push-forward of $\nu |_Q$ by ${\bf T}_Q$. Thus, $\nu^Q$ is always a probability measure supported on the closure of ${\bf T}_Q(\supp \nu \cap Q)$, a subset of $[0,1]^d$:
\begin{equation} \label{push-forward measure} \nu^Q(E) = \frac{\nu({\bf T}_Q^{-1}(E) \cap Q)}{\nu(Q)} \quad \text{ for any Borel set } E \subseteq \mathbb R^d. \end{equation}  

\begin{proposition}\label{existence of measure}
For each $\mathtt N \geq d$, there exist admissible constants $\mathtt A, \mathtt B, \mathtt C$ obeying \eqref{AB conditions} and an admissible constant $\varepsilon = \varepsilon(\mathtt A, \mathtt B, \mathtt C) > 0$ with the following property:  Let $\Phi\colon [0,1] \rightarrow \mathbb R^d$ be any smooth function in standard form that is vanishing of type $\mathtt N$ at the origin, and let $K \subseteq \R^d$ be any Borel set with $\dim_{\textnormal H} K > d-\varepsilon$. Then there exists
\begin{itemize}
\item{a dyadic box $\mathtt{Q} \in \fD_{\mathtt J}^{\ast}[\vec{\mathtt n}]$ with $\mathtt J = \mathtt{J}(\Phi)$ as in Proposition \ref{spectral gap lemma} and $\vec{\mathtt n} = (\mathtt n_1,\ldots,\mathtt n_d)$ as in \eqref{ordering}, and}
\item{a finite Borel measure $\nu$ with $\supp\nu \subseteq K \cap \overline{\mathtt{Q}}$ and $\nu(\mathtt Q) > 0$}
\end{itemize}
such that the blow-up $\mu \colonequals \nu^{\mathtt{Q}}$ satisfies the energy and spectral gap conditions in \eqref{spectral gap condition}. 
\end{proposition}
Here, $\overline{E}$ denotes the topological closure of $E$.

\subsection{Proof of Theorems \ref{uniform result} and \ref{finite type result}, assuming Propositions \ref{existence of pattern}--\ref{existence of measure}}

The proof is a concatenation of the three propositions, in reverse order.   Let $\Gamma \subset \R^d$ be a smooth curve that is vanishing of type $\mathtt N$ at the origin.  Let $K \subseteq \R^d$ be a Borel set with $\dim_{\mathrm H} K > d-\varepsilon$, with $\varepsilon$ as in Proposition \ref{existence of measure}.   We may assume that $\Gamma \supseteq \Phi([0,1])$, where $\Phi \colon [0,1] \rightarrow \R^d$ is in standard form and vanishing of type $\mathtt N$ at the origin.

Let $\mathtt{Q}$ and $\mu$ be as in the conclusion of Proposition \ref{existence of measure} when applied to $\Phi$ and $K$.  Thus, in particular, $\mu$ is a probability measure supported on ${\bf T}_{\mathtt{Q}}(K \cap \overline{\mathtt{Q}}) \subseteq [0,1]^d$ that satisfies the criteria \eqref{spectral gap condition} for some constants $\mathtt A, \mathtt B, \mathtt C$ obeying \eqref{AB conditions}.  Let $\mathtt j \geq \mathtt J(\Phi)$ be such that $\mathtt{Q} \in \fD_{\mathtt j}$.  Proposition \ref{spectral gap lemma} gives that $\mathscr{I}[\mu, \pi[\Phi; \mathtt{j}, \mathtt{A}^{-6d}]] > 0$, and consequently, by Proposition \ref{existence of pattern} there exists some $x \in \supp \pi[\Phi; \mathtt{j}, \mathtt{A}^{-6d}] \cap (\supp \mu - \supp \mu)$.

The measure $\pi[\Phi; \mathtt{j}, \mathtt{A}^{-6d}]$ is supported on $\Phi^{\mathtt{j}}([\mathtt{A}^{-6d},1])$, where $\Phi^{\mathtt{j}}$ is as in \eqref{rescaled Phi}, while $\mu$ is supported on ${\bf T}_{\mathtt{Q}}(K)$.  Writing $x = \Phi^{\mathtt{j}}(s)$ for some $s \in [\mathtt{A}^{-6d},1]$ and setting $\gamma \colonequals \Phi(2^{-\mathtt{j}}s)$, it follows from \eqref{nonzero property} and \eqref{rescaling map} that $\gamma \in \Gamma \setminus \{0\}$ and $\gamma \in K-K$. \qed

\begin{remark}
The above argument relies crucially on the relationship between the curve $\Gamma$ and the collection $\fD^* = \fD^*[\vec{\mathtt n}]$ of dyadic boxes.  We would like to explain how these boxes were chosen.  For simplicity, we will consider the example of $\Gamma = \im \Phi$ with
\begin{align*}
\Phi(t) = (t^2, t^3+t^4), \quad\quad t \in [0,1],
\end{align*}
and thus $\vec{\mathtt n} = (2,3)$.  Roughly speaking, Propositions \ref{spectral gap lemma} and \ref{existence of pattern} imply that if $K$ supports a measure with a spectral gap, then $K-K$ intersects $\Gamma \setminus \{0\}$ as desired.  Proposition \ref{existence of measure} provides a rectangle $\mathtt{Q} \in \fD^*$ such that the affine image ${\bf T}_{\mathtt{Q}}(K)$ of $K$ under the rescaling map for $\mathtt{Q}$ supports a measure with this property.  In an ideal scenario, the dimensions of our dyadic rectangles in $\fD^*$ would be chosen so that $\Gamma$ would be invariant under the (linear part of the) inverse scaling map ${\bf T}_{\mathtt{Q}}^{-1}$. This would then allow us to pull back the desired pattern from ${\bf T}_{\mathtt{Q}}(K)$ to $K$.  However, the curve $\Gamma$ in the present example does not enjoy any such scaling relation.  Instead, it possesses an approximate scaling relation based on its leading order behaviour:  The rescaled functions
\begin{align*}
\Phi^j(t) \colonequals (2^{2j}\Phi_1(2^{-j}t), 2^{3j}\Phi_2(2^{-j}t)) = (t^2, t^3 + 2^{-j}t^4)
\end{align*}
yield a sequence of curves $\Gamma_j \colonequals \im \Phi^j$ that approach the ``leading order'' curve
\begin{align*}
\Gamma_\infty \colonequals \{(t^2, t^3) \colon t \in [0,1]\}
\end{align*}
as $j \rightarrow \infty$, and this limit curve is invariant under the transformation $(x_1,x_2) \mapsto (2^{2j}x_1,2^{3j}x_2)$ for any $j$.  We define the rectangles in $\fD_j$ to have dimensions $2^{-2j} \times 2^{-3j}$, so that (the linear parts of) their rescaling maps coincide with this transformation.  Since Proposition \ref{spectral gap lemma} is in fact uniform in $j$ (sufficiently large), we can show that ${\bf T}_{\mathtt{Q}}(K)-{\bf T}_{\mathtt{Q}}(K)$ contains a nonzero point $\Phi^{\mathtt{j}}(s)$, where $\mathtt{j}$ is such that $\mathtt{Q} \in \fD_{\mathtt{j}}$.  As desired, the rescalings $K \rightarrow {\bf T}_{\mathtt{Q}}(K)$ and $\Phi \rightarrow \Phi^{\mathtt{j}}$ are compatible, in the sense that now $K-K$ must contain the nonzero point $\Phi(2^{-\mathtt{j}}s) \in \Gamma$.
\end{remark}

\section{High-dimensional graphs} \label{sec3}
The proofs of Theorems \ref{graph result} and \ref{polynomial result} require construction of counterexamples to unavoidability.  For these, we will utilize the existence of high-dimensional one-parameter graphs. 

\begin{proposition}\label{high dimensional graph}
For every $s \in [1,d)$, there exists a H\"older-continuous function $F_s \colon [0,1] \rightarrow \R^{d-1}$ such that 
\begin{align}
&\|F_{s}\|_{C^{0, \alpha}} < \infty \quad \text{for all} \quad  0 < \alpha < \min\Big\{\frac{1}{s}, \frac{d-s}{d-1}\Big\},\label{Holder}
\intertext{and}
&\dim_{\mathrm{H}}(\graph F_s) = s, \quad \text{where} \quad \graph F_s \colonequals \{(t, F_s(t)) \colon t \in [0,1]\};\label{counterexample}
\end{align}
here, $\| \cdot\|_{C^{0, \alpha}}$ denotes the H\"older norm 
\[ \|f\|_{C^{0, \alpha}} \colonequals \sup_{t \in [0,1]} |f(t)| + \sup_{t,t' \in [0,1] \colon t \neq t'} \frac{|f(t) - f(t') |}{|t-t'|^{\alpha}}. \] 
\end{proposition}
This classical result, originally due to Besicovitch and Ursell \cite{Besicovitch-Ursell} for $d=2$, now has many variants in the literature; see for example \cite{{Love-Young},{Kline}, {Levy-1}, {Levy-2}, {Taylor}, {Yoder}}. It has been proved in its above-stated form by Kahane \cite{Kahane-Book}, who shows through a random argument that such functions $F_{s}$ are in fact plentiful. We briefly outline his argument, pointing the reader to the relevant sections of the text for a complete proof.

In \cite[Chapter 18]{Kahane-Book} and for $n, d \geq 1$, Kahane introduces an $(n, d, \gamma)$ Gaussian process $\{X_t \colon t \in \mathbb R^n\}$ with values in $\mathbb R^d$ such that 
\[ \mathbb E(| X_t - X_{t'}|^2) = d|t-t'|^{\gamma}. \]
Such a process is shown to exist in \cite[Chapter 18, \S 2]{Kahane-Book} when $0 < \gamma \leq 2$, with an almost sure continuous version, i.e.~with $t \mapsto X(t; \omega)$ being a continuous function of $t$ for almost every $\omega$. More precisely, the modulus of continuity $\omega_{X}$ of $X(\cdot;\omega)$ is shown to obey 
\begin{equation} \label{Holder-exponent} \omega_{X}(h) =  \sup_{|t-t'| \leq h} |X(t) - X(t')| = O \bigl(\sqrt{|h|^{\gamma} \log(1/|h|)} \bigr) \qquad \text{ almost surely }  \end{equation}
on every compact subset of $\mathbb R^n$. This is stated in equation (3) of 
\cite[Chapter 18]{Kahane-Book}, and follows from the content of \cite[Chapter 14]{Kahane-Book}. The condition \eqref{Holder-exponent} implies that (almost surely)
\begin{align}\label{Holder continuity of X(t)}
\text{$X(t)$ is H$\ddot{\text{o}}$lder continuous with exponent $\gamma/2 - \varepsilon$ for every $\varepsilon \in (0, \gamma/2)$.}
\end{align}
In \cite[Chapter 18, \S 7, Theorem 7]{Kahane-Book}, Kahane proves:

{\em{For any compact set $E \subset \mathbb R^n$,  the relation }}
\[\dim_\mathrm{H} \bigl(\operatorname{graph} X|_E \bigr) = 
\min \Bigl\{ \frac{2}{\gamma} \dim_\mathrm{H} E, \, \dim_\mathrm{H} E + \Big(1 - \frac{\gamma}{2}\Big)d \Bigr\} \]
{\em{holds almost surely}}.

Here, $\graph X|_E$ denotes the set $\{(t, X(t)) \colon t \in E \}$.  We can now obtain Proposition \ref{high dimensional graph} from Kahane's theorem as follows:  Set $n=1$ and $E = [0,1]$, and replace $d$ by $(d-1)$ and $X(t)$ by $F_{s}(t)$ to get
\begin{align}\label{dim of F_s}
\dim_\mathrm{H}(\graph F_s) = \dim_\mathrm{H}(\operatorname{graph} X|_{[0,1]}) = \min \Bigl\{\frac{2}{\gamma}, d- \frac{\gamma}{2}(d-1)  \Bigr\},
\end{align}
If we choose
\begin{align*}
\gamma = \min\Big\{\frac{2}{s}, \frac{2(d-s)}{d-1}\Big\} =  \begin{cases}
\frac{2}{s} &\text{if } 1 \leq s \leq d-1,\\
\frac{2(d-s)}{d-1} &\text{if } d-1 < s < d,
\end{cases}
\end{align*}
then \eqref{Holder continuity of X(t)} implies the first conclusion of the proposition, and \eqref{dim of F_s} and a bit of arithmetic confirm the second.  \qed

\section{Graphs of infinite type are avoidable}  \label{sec4}
The goal of this section is to prove Theorem \ref{graph result}. We will in fact prove a slightly stronger statement (Proposition \ref{generalized-graph-lemma} below), namely an analogue of the theorem for curves that are not necessarily graphs, but graph-like.  Additionally, we formulate a quantitative partial-avoidance result for graph-like curves of finite type (Proposition \ref{partial-avoidance graphs}); the proof is sketched in the Appendix.

Let us set up the relevant definition.  We say that a smooth curve $\Gamma \subset \R^d$ is {\em{graph-like}} if there exists an invertible linear transformation ${\bf L} \colon \R^d \rightarrow \R^d$ and a smooth function $\Phi \colon \mathtt I \rightarrow \R^d$ such that
\begin{enumerate}[(i)]
\item{$\mathtt I \subset \R$ is a nondegenerate compact interval,}
\item{$\Phi =: (\Phi_1, \underline{\Phi})$ is of the form $\Phi_1(t) = t^{\mathtt m}\phi(t)$ for some integer $\mathtt m \geq 1$ and some smooth function $\phi \colon \mathtt I \rightarrow \R$ such that $\inf\{|\phi(t)| \colon t \in \mathtt I\} > 0$,}
\item{${\bf L}(\Gamma) = \Phi(\mathtt I)$.}
\end{enumerate}
In the above, if ${\bf L}$ is the identity, $\mathtt m = 1$, and $\phi \equiv 1$, then $\Gamma = \{(t, \underline{\Phi}(t)) \colon t \in \mathtt I\}$ is an ordinary graph.

\begin{proposition} \label{generalized-graph-lemma}
Let $\Gamma \subset \R^d$ be a smooth graph-like curve that contains the origin.  Then $\Gamma$ is unavoidable if and only if $\Gamma$ is of finite type at the origin.  In particular, Theorem \ref{graph result} holds.
\end{proposition}

\subsection{Proof of Proposition \ref{generalized-graph-lemma}}

Theorem \ref{finite type result} provides one direction of the proposition, namely that if $\Gamma$ is of finite type at the origin, then $\Gamma$ is unavoidable.  Toward proving the other direction, we assume that $\Gamma$ is of infinite type at the origin and aim to show that $\Gamma$ is avoidable.  Fix ${\bf L}$ and $\Phi \colon \mathtt I \rightarrow \R^d$ satisfying conditions (i)--(iii) in the definition of graph-like curve.  We may assume that ${\bf L}$ is the identity, since unavoidability and type are both invariant under the action of invertible linear transformations.  The hypothesis that  $\Gamma$ contains the origin, together with conditions (ii) and (iii), implies that $0 \in \mathtt I$ and that $\Phi(0) = 0$.  Moreover, $\Phi$ must be of infinite type at the origin, since we have assumed the same of $\Gamma$.  Consequently, there exists a unit vector $\mathtt u \in \R^d$ such that
\begin{align}\label{infinite type condition}
\mathtt u \cdot \Phi^{(n)}(0) = 0 \quad \text{for every } n \geq 0.
\end{align}
Let ${\bf U} \colon \R^d \rightarrow \R^d$ be the unitary matrix that maps $\mathtt{e}_1 = (1,0,\ldots, 0)$ to $\mathtt u$.  Define $z \colonequals {\bf U}^{-1} \circ \Phi$, and write $z = (z_1, \underline{z})$ with $z_1 \colon \mathtt I \rightarrow \R$ and $\underline{z} \colon \mathtt I \rightarrow \R^{d-1}$.  Then \eqref{infinite type condition} implies that
\begin{align}\label{z_1 derivative condition}
z_1^{(n)}(0) = [\mathtt{e}_1 \cdot ({\bf U}^{-1} \circ \Phi)]^{(n)}(0) = \mathtt u \cdot \Phi^{(n)}(0) = 0 \quad \text{for every } n \geq 0.
\end{align}
It follows that for each $n$, there exists $\mathtt C_n > 0$ such that
\begin{align}\label{z_1 upper bound}
|z_1(t)| \leq \mathtt C_n|t|^n \quad \text{for all } t \in \mathtt I.
\end{align}
By condition (ii), we have $\Phi^{(\mathtt m)}(0) \neq 0$ and thus
\begin{align}\label{z derivative condition}
z^{(\mathtt m)}(0) = {\bf U}^{-1} \circ \Phi^{(\mathtt m)}(0) \neq 0.
\end{align}
Properties \eqref{z_1 derivative condition} and \eqref{z derivative condition} together imply that $\underline{z}^{(\mathtt m)}(0) \neq 0$.  It follows that there exist constants $\mathtt c_0 > 0$ and $\delta > 0$ such that
\begin{align}\label{z bar lower bound}
|\underline{z}(t)| \geq \mathtt c_0|t|^{\mathtt m} \quad \text{for all } t \in \mathtt I \cap [-\delta,\delta].
\end{align}
Now, fix $s \in [1,d)$ and consider the graph of $F_s$, as defined in \eqref{counterexample}.  Given any $r > 0$, we can find a cube $Q_r \subset \R^d$ of diameter $r$ such that
\begin{align*}
\dim_\mathrm{H}({\bf U}(\graph F_s) \cap Q_r) = \dim_\mathrm{H}({\bf U}(\graph F_s)) = \dim_\mathrm{H} (\graph F_s) = s.
\end{align*}
Fix $\mathtt r$ such that
\begin{align*}
0 < \mathtt r < \mathtt c_1\min\bigg\{\delta, \bigg(\frac{\mathtt c_0}{\|F_s\|_{C^{0,\alpha}}\mathtt{C}_{\mathtt n}^\alpha}\bigg)^\frac{1}{\mathtt{n}\alpha - \mathtt m}\bigg\}^{\mathtt m},
\end{align*}
where
\begin{itemize}
\item{$\mathtt c_1 \colonequals \inf\{|\phi(t)| \colon t \in \mathtt I\} > 0$,}
\item{$\alpha$ is any exponent satisfying \eqref{Holder},}
\item{$\|F_s\|_{C^{0,\alpha}}$ is the H\"older norm of $F_s$, as in \eqref{Holder},}
\item{$\mathtt n$ is any integer such that $\mathtt{n}\alpha > \mathtt m$.}
\end{itemize}
Let
\begin{align*}
K \colonequals {\bf U}(\graph F_s) \cap Q_{\mathtt r}.
\end{align*}

We claim that $(\Gamma \setminus \{0\}) \cap (K - K) = \emptyset$.  For a contradiction, suppose there exists some $\gamma \in (\Gamma \setminus \{0\}) \cap (K-K)$.  Because $\gamma$ is a member of $\Gamma$, we can express it as $\gamma = \Phi(\mathtt{t})$ for some $\mathtt{t} \in \mathtt I$.  This $\mathtt{t}$ obeys
\begin{align*}
\mathtt c_1|\mathtt{t}|^{\mathtt m} \leq |\Phi_1(\mathtt{t})| \leq |\Phi(\mathtt{t})| = |\gamma| \leq \diam K \leq \diam Q_{\mathtt r} = \mathtt r,
\end{align*}
using that $\gamma \in K-K$.  Thus,
\begin{align}\label{t_0 upper bound}
|\mathtt{t}| < \min\bigg\{\delta, \bigg(\frac{\mathtt c_0}{\|F_s\|_{C^{0,\alpha}}\mathtt{C}_{\mathtt n}^\alpha}\bigg)^\frac{1}{\mathtt{n}\alpha - \mathtt m}\bigg\}
\end{align}
due to our choice of $\mathtt r$.  Using the graph structure of $K$, we can also express $\gamma$ as $\gamma = {\bf U}(t-t', F_s(t)-F_s(t'))$ for some $t,t' \in [0,1]$.  This leads to the relation
\begin{align*}
z(\mathtt{t}) = (t-t', F_s(t)-F_s(t')).
\end{align*}
By \eqref{t_0 upper bound}, \eqref{z bar lower bound}, \eqref{z_1 upper bound}, and the H\"older continuity of $F_s$, we have
\begin{align*}
\mathtt c_0|\mathtt{t}|^{\mathtt m} \leq |\underline{z}(\mathtt{t})| = |F_s(t)-F_s(t')| &\leq \|F_s\|_{C^{0,\alpha}}|t-t'|^\alpha\\ &= \|F_s\|_{C^{0,\alpha}}|z_1(\mathtt{t})|^\alpha \leq \|F_s\|_{C^{0,\alpha}}\mathtt{C}_{\mathtt n}^\alpha|\mathtt{t}|^{\mathtt{n}\alpha}.
\end{align*}
This inequality is compatible with \eqref{t_0 upper bound} only if $\mathtt{t} = 0$.  So $\gamma = \Phi(0) = 0$, contradicting the assumption that $\gamma \in \Gamma \setminus \{0\}$.  Since $\dim_\mathrm{H}K =s$ and $s \in [1,d)$ was arbitrary, we conclude that $\Gamma$ is avoidable. \qed

\subsection{A quantitative partial-avoidance result}\label{partial-avoidance graphs}\label{quantitative partial-avoidance}
The proof of Proposition \ref{generalized-graph-lemma} can be modified slightly to give a quantitative partial-avoidance result for graph-like curves of finite type.  To formulate such a statement, we need another definition.  Let $\Gamma$ be a graph-like curve.  There are many choices of ${\bf L}$, $\mathtt I$, and $\Phi(t) = (t^\mathtt{m}\phi(t), \underline{\Phi}(t))$ such that conditions (i)--(iii) in the definition of graph-like curve are satisfied.  The smallest integer $\mathtt m$ appearing among these parametrizations will be called the \emph{subtype} of $\Gamma$.  We sketch a proof of the following in the Appendix:
\begin{proposition}

Let $\Gamma$ be a graph-like curve of type $\mathtt N$ at the origin and of subtype $\mathtt m$.  Then for every $s < \min\big\{\frac{\mathtt N}{\mathtt m}, d-\frac{(d-1)\mathtt m}{\mathtt N}\big\}$, there exists a Borel set $K \subseteq \R^d$ with $\dim_{\mathrm H} K \geq s$ such that ${(\Gamma \setminus \{0\}) \cap (K - K) = \emptyset}$.
\end{proposition}

Taking $\Gamma$ to be a curve of subtype $\mathtt m = 1$, this result implies that the constant $\varepsilon_{\mathtt N}$ in Theorem \ref{uniform result} cannot exceed $\frac{d-1}{\mathtt N}$.

\section{Polynomial patterns of infinite type are avoidable} \label{sec5}
In this section, we prove Theorem \ref{polynomial result}, which asserts equivalence between five statements.  We will only demonstrate the equivalence of statements 1 and 2, namely that $\Gamma$ is unavoidable if and only if $\Gamma$ is of finite type at the origin.  It is straightforward to show that statements 2--5 are equivalent; this is left to the reader.  

\subsection{Proof of Theorem \ref{polynomial result}}
Fix a polynomial curve $\Gamma \subset \R^d$ that contains the origin and a polynomial function $\Phi \colon \mathtt I \rightarrow \R^d$ such that $\Gamma = \Phi(\mathtt I)$.  We may assume that $\Phi(0) = 0$.  As in the proof of Proposition \ref{generalized-graph-lemma}, one direction of the equivalence between statements 1 and 2 is already supplied by Theorem \ref{finite type result}, namely that 2 implies 1.  We therefore assume that $\Gamma$ is of infinite type at the origin and aim to show that $\Gamma$ is avoidable; this would show that 1 implies 2.  Toward that end, we note that the parametrization $\Phi$ must be of infinite type at the origin and, because $\Phi$ is a polynomial function, this is equivalent to the existence of a unit vector $\mathtt u \in \R^d$ such that $\mathtt u \cdot \Phi \equiv 0$.  Let ${\bf U} \colon \R^d \rightarrow \R^d$ be the unitary matrix that maps $\mathtt{e}_1 = (1,0,\ldots,0)$ to $\mathtt u$.  Fix $s \in [1,d)$, and let
\begin{align*}
K \colonequals {\bf U}(\graph F_s),
\end{align*}
where $\graph F_s$ is as in \eqref{counterexample}.  We claim that $(\Gamma \setminus \{0\}) \cap (K-K) = \emptyset$.  For a contradiction, suppose there exists $\gamma \in (\Gamma \setminus \{0\}) \cap (K-K)$, and let $z = {\bf U}^{-1}(\gamma)$.  Writing $\gamma = \Phi(\mathtt{t})$ for some $\mathtt{t} \in \mathtt I$, we have
\begin{align}\label{z_1 equals zero}
z_1 = \mathtt{e}_1 \cdot z = {\bf U}^{-1}\mathtt u \cdot z = \mathtt u \cdot {\bf U}z =  \mathtt u \cdot \Phi(\mathtt{t}) = 0.
\end{align}
Due to the graph structure of $K$, we also have
\begin{align}\label{z graph structure}
z = {\bf U}^{-1}(\gamma) = (t-t', F_s(t)-F_s(t'))
\end{align}
for some $t,t' \in [0,1]$.  Together, \eqref{z_1 equals zero} and \eqref{z graph structure} imply that $z = 0$ and hence $\gamma = 0$, a contradiction.  Since $\dim_\mathrm{H}K = s$ and $s \in [1,d)$ was arbitrary, we conclude that $\Gamma$ is avoidable. \qed

\subsection{Checking conditions for (un)avoidability} \label{linear independence conditions}
As noted in Section \ref{Introduction}, an unavoidable smooth curve must contain the origin.  Combining this observation with statements 4 and 5 in Theorem \ref{polynomial result}, we get a simple criterion for checking whether a given tuple of polynomials $\Phi = (\Phi_1,\ldots, \Phi_d)$ defines an unavoidable curve on a compact interval $\mathtt I$:  The image of $\Phi$ on $\mathtt I$ is unavoidable if and only if $\Phi_1,\ldots,\Phi_d$ are linearly independent and share a common zero in $\mathtt I$.  The following examples illustrate this with $d = 3$:
\begin{itemize}
\item{$\Gamma = \{(t^2-1, t^3+5t-6, 2t^3-t^2-1)\colon t \in [0,1]\}$ is unavoidable.  The parametrizing polynomials are linearly independent and vanish at $t=1$.}
\item{$\Gamma = \{(t-2, t^2-2t, t^2+t-6) \colon t \in [0,1]\}$ is avoidable.  The parametrizing polynomials are linearly dependent.}
\item{$\Gamma = \{(t+1, t^2-1, t^3+2t+1) \colon t \in [0,1]\}$ is avoidable.  The parametrizing polynomials do not share a zero.}
\end{itemize}
In general, the presence of a shared zero among polynomials can be checked by, say, using the Euclidean algorithm to compute their greatest common divisor and then using Sturm's theorem to determine whether that divisor has a real zero.

\section{Configuration integral:  Proof of Proposition \ref{existence of pattern}} \label{nonvanishing integral section}

In this section, we prove Proposition \ref{existence of pattern}.  Let $\mathscr{I}_0 \colonequals \mathscr{I}[\mu,\pi]/2 > 0$, with $\mathscr{I}[\mu,\pi]$ as in \eqref{configuration integral}.  Then there exists $\delta_0 > 0$ such that
\begin{equation} \label{positivity to integral} 
\Big\vert\int (\mu \ast \psi_\delta)\ast\pi\,d\mu\Big\vert \geq \mathscr{I}_0 \quad \text{for all } \delta \in (0,\delta_0].
\end{equation}
Since $\mu$ is a probability measure, property \eqref{positivity to integral} implies that for every $\delta \in (0, \delta_0]$, there exists a point $x_\delta \in \supp \mu$ such that
\begin{align*}
|(\mu \ast \psi_\delta) \ast \pi(x_\delta)| \geq \mathscr{I}_0. 
\end{align*}
Set $\pi_\delta \colonequals \pi \ast \psi_\delta$. Since $(\mu \ast \psi_\delta) \ast \pi = \mu \ast \pi_\delta$ by properties of convolution, we obtain
\begin{align}\label{K1K2}
\mathscr{I}_0 \leq |\mu \ast \pi_\delta(x_\delta)| &= \Big\vert\int\pi_\delta(x_\delta-y)d\mu(y)\Big\vert \nonumber \\ 
& = \int_{E_1}|\pi_\delta(x_\delta-y)|d\mu(y) + \int_{E_2}|\pi_\delta(x_\delta-y)|d\mu(y),
\end{align}
where
\begin{align}
E_1 = E_1(\delta) &\colonequals \bigl\{ y  \colon \dist(x_\delta-y, \supp \pi)>\sqrt{\delta} \bigr\}, \nonumber \\ 
E_2 = E_2(\delta)  &\colonequals \bigl\{ y  \colon \dist(x_\delta-y, \supp \pi) \leq \sqrt{\delta} \bigr\}. \nonumber 
\end{align}
We claim that the integral over $E_1$ vanishes as $\delta \rightarrow 0$.  Indeed, fix $z \colonequals x_\delta - y$ such that $\dist(z,\supp \pi) > \sqrt{\delta}$. Then for each integer $N$, there exists $C_N < \infty$ such that
\begin{align*}
|\pi_\delta(z)| \leq \delta^{-d}\int|\psi(\delta^{-1}(z-w))|d\pi(w) &\leq C_N \delta^{-d}\int(\delta^{-1}|z-w|)^{-N}d\pi(w) \leq C_N \|\pi\| \delta^{\frac{N}{2} - d};
\end{align*}
here, we have used the rapid decay of $\psi$.  Noting that $\|\pi\| < \infty$ and taking $N > 2(d+1)$ and $\delta$ sufficiently small, we get that $|\pi_\delta(z)| \leq \delta$.  This pointwise bound on the integrand means that
\begin{equation} \label{K1-estimate} 
\int_{E_1}|\pi_\delta(x_\delta-y)|d\mu(y) \leq \delta.
\end{equation}
Now, assuming $\delta < \mathscr{I}_0/2$ and inserting \eqref{K1-estimate} into \eqref{K1K2}, we find that 
\begin{align*} 
\int_{E_2}|\pi_\delta(x_\delta-y)|d\mu(y) \geq \frac{\mathscr{I}_0}{2} > 0.
\end{align*}
Hence, for each sufficiently small $\delta$, there exists $y_\delta \in \supp \mu$ such that
\begin{align}\label{distance from support of pi}
\dist(x_\delta - y_\delta, \supp \pi) \leq \sqrt{\delta}.
\end{align}
Since $\supp(\mu) \times \supp(\mu)$ is a compact set in $\R^{2d}$, there exists a sequence of values of $\delta$ along which $(x_\delta,y_\delta)$ converges to a point $(x,y) \in \supp(\mu) \times \supp(\mu)$.  By \eqref{distance from support of pi}, we must have $x-y \in \supp \pi$, and the conclusion of the proposition follows. \qed

\section{Energy and spectral gap:  Proof of Proposition \ref{spectral gap lemma}} \label{energy and spectral gap section}

The goal of this section is to prove Proposition \ref{spectral gap lemma}. Essential to this proof is a deeper understanding of the behaviour of the measures $\pi = \pi[\Phi; j, c]$ defined for functions $\Phi\colon [0,1] \rightarrow \mathbb R^d$ in standard form, with special attention to their dependence on the accompanying parameters $j$ and $c$. We are specifically interested in the growth rate of the mass assigned by $\pi$ to Euclidean balls and the decay of its Fourier transform $\widehat{\pi}$. We collect the main tools in the first three subsections. Using these, the proof of  Proposition \ref{spectral gap lemma} is completed in Subsection \ref{spectral gap lemma proof}. 

\subsection{Choice of constants} \label{CN J0}
Let $\Phi \colon [0,1] \rightarrow \R^d$ be a smooth function in standard form that is vanishing of type $\mathtt N$ at the origin. The discussion in Subsection \ref{spectral gap proposition section} leading up to \eqref{c1} and \eqref{c2} identifies two constants $\mathtt K_{\mathtt N} \colonequals 2\mathtt N!$ and $\mathtt J_0 = \mathtt{J}_0(\Phi)$, only the former being admissible. These two constants are important for our subsequent analysis: the admissible constant $\mathtt{L}_{\mathtt N}$ and the inadmissible constant $\mathtt J$ appearing in Proposition \ref{spectral gap lemma} will depend respectively on $\mathtt K_{\mathtt N}$ and $\mathtt J_0$. In the remainder of this section, $\mathtt{L}_{\mathtt N}$ will always denote an admissible constant and $\mathtt J$ an inadmissible one, although their exact values may change from one occurrence to another. In particular, $\mathtt{L}_{\mathtt N}$ will always be a large multiple of $\mathtt K_{\mathtt N}$. The multiplicative factor may depend on $d$, $\mathtt N$, and the Schwartz function $\psi$ introduced in \eqref{def psi} in order to define $\mu_{\delta}$.

\subsection{Ball condition for $\pi$}
\begin{lemma} \label{pi ball lemma}
Fix any $\mathtt{L}_{\mathtt N} \geq 2d\mathtt{K}_\mathtt{N}$, where $\mathtt K_{\mathtt N}$ is the constant defined in Subsection \ref{CN J0}.  Then for every $\Phi$ in standard form (of type $\mathtt N$ at the origin), we have
\begin{align}\label{ball lower bound}
\inf \Bigl\{ \pi(B(0;r)) \colon \pi = \pi[\Phi; j, c],~j \geq \mathtt J_0(\Phi), ~c \in \Big(0, \frac{r}{\mathtt{L}_{\mathtt N}}\Big] \Bigr\} \geq \frac{r}{\mathtt{L}_{\mathtt N}} \quad \text{for every $r \in (0,1]$.}
\end{align}
\end{lemma}  
\begin{proof} 
The condition \eqref{c1} with $\ell=0$ gives for all $j \geq \mathtt J_0$ and $s \in [0, 1]$ the bound 
\begin{align}\label{Phi linear} 
|\Phi^j(s)| \leq \sum_{i=1}^{d} 2^{\mathtt n_ij} |\Phi_i(2^{-j}s) | \leq \sum_{i=1}^{d} 2^{\mathtt n_ij}  \mathtt K_{\mathtt N} (2^{-j}s)^{\mathtt n_i} \leq \mathtt K_{\mathtt N} \sum_{i=1}^{d} s^{\mathtt n_i} \leq \mathtt K_{\mathtt N} d s \leq \frac{\mathtt{L}_{\mathtt N}}{2}s.
\end{align} 
The fourth inequality in the display above uses the relation $\mathtt n_i \geq 1$, a consequence of \eqref{ordering}. This upper bound allows us to estimate $\pi(B(0;r))$ as follows. For $r \in (0,1]$ and $c \in (0,r/\mathtt{L}_{\mathtt N}]$, we have
\begin{align*}
\pi (B(0;r)) = \int_{c}^{1} \mathbf 1_{B(0;r)}(\Phi^j(s)) ds  &= | \{s \in [c,1] \colon |\Phi^j(s)| \leq r \}|\\
&\geq \Big\vert\Big\{s \in [c,1] \colon \frac{\mathtt{L}_{\mathtt N}}{2}s \leq r\Big\}\Big\vert  \quad \text{(from \eqref{Phi linear})}\\
&\geq \frac{2r}{\mathtt{L}_{\mathtt N}} - c \geq \frac{r}{\mathtt{L}_{\mathtt N}},
\end{align*}
where $|E|$ denotes the Lebesgue measure of $E$. This gives the desired conclusion \eqref{ball lower bound}.
\end{proof}

The ball condition on $\pi$ permits size estimates on certain convolutions of measures involving $\pi$. The following corollary in particular will be helpful in proving Proposition \ref{spectral gap lemma}, since it offers a pointwise bound on the integrand of the configuration integral $\mathscr{I}$.  Let us recall the definition $\mu_h = \mu \ast \psi_{h}$ from \eqref{def-mu-psi-delta} with $\psi$ as in \eqref{def psi}.
\begin{corollary} \label{convolution estimate corollary} 
Let $\mathtt a > 0$ be an absolute constant.  There exists an admissible constant $\mathtt{L}_{\mathtt N} > 0$, depending only on $\mathtt N$, $\mathtt{a}$, and the auxiliary function $\psi$ chosen in \eqref{def psi}, with the following property:  Let
$\mu$ be a Borel measure on $\mathbb R^d$ and $x_0 \in \mathbb R^d$ a point such that
\begin{equation} \label{lower bound ball pi}
\inf_{r \in (0,1]} \frac{\mu(B(x_0;r))}{r^{d}} \geq \mathtt a.
\end{equation}  
Then for every $\Phi$ in standard form (of type $\mathtt N$ at the origin), we have
\begin{equation}  \mu_{h} \ast \pi (x_0) \geq \mathtt{L}_{\mathtt N}^{-1} h \quad \text{for all $\pi = \pi[\Phi; j, h^2]$ with $j \geq \mathtt J_0$ and $h \in (0, \mathtt{L}_{\mathtt N}^{-1}]$}. \label{convolution lower bound} \end{equation}     
\end{corollary} 
\begin{proof}
Fix a Borel measure $\mu$ and a point $x_0$ satisfying \eqref{lower bound ball pi}.  According to \eqref{conditions-psi}, the function $\psi$ satisfies $\psi(0) = 1$. Therefore, there exists an absolute constant $\mathtt{b} > 0$ such that $\psi(x) \geq 1/2$ for $|x| \leq \mathtt{b}$.  Since $\psi$ is nonnegative, it follows that
\[ \mu_{h}(y) = h^{-d} \int \psi(h^{-1}(y-z))d\mu(z) \geq \frac{1}{2h^d} \mu(B(y; \mathtt{b}h)). \] 
If $y$ is any point such that $|x_0-y| \leq \mathtt{b}h/2$, then $B(x_0; \mathtt{b}h/2) \subseteq B(y;\mathtt{b}h)$. Inserting this inclusion into the estimate above yields 
\begin{align*}
 \mu_{h}(y)  \geq \frac{1}{2h^d}\mu(B(x_0; \mathtt{b}h/2)) \geq \mathtt c \quad \text{for all } y \in B(x_0; \mathtt{b}h/2), 
\end{align*}
where $\mathtt c \colonequals \mathtt a \mathtt{b}^d 2^{-d-1} > 0$. Assume $h > 0$ is small enough that $h^2 \leq (\mathtt{b}h/2)/(2d \mathtt K_{\mathtt N})$, i.e.~$h \leq \mathtt{b}/(4d \mathtt K_{\mathtt N})$. Then Lemma \ref{pi ball lemma} applies with the quantities $c$, $r$, and $\mathtt{L}_{\mathtt N}$ there being replaced by $h^2$, $\mathtt{b}h/2$, and $2d\mathtt K_{\mathtt N}$ respectively.  If $\pi$ is of the form $\pi = \pi[\Phi; j, h^2]$, it follows  from \eqref{ball lower bound} that
\begin{align*}
\mu_{h} \ast \pi(x_0) = \int\mu_{h}(x_0-y)d\pi(y) &\geq \int_{B(0; \mathtt{b}h/2)}\mu_{h}(x_0-y)d\pi(y) \\ 
&\geq \mathtt c \pi(B(0;{\mathtt{b}h}/{2})) \geq  \frac{\mathtt{b} \mathtt c}{4d \mathtt K_{\mathtt N}}h \geq \mathtt{L}_{\mathtt N}^{-1} h,  
\end{align*}
provided $\mathtt{L}_{\mathtt N} \geq 4d \mathtt K_{\mathtt N}/(\mathtt{b} \mathtt c)$.  Assuming also that $\mathtt{L}_{\mathtt N} \geq 4d \mathtt K_{\mathtt N}/\mathtt{b}$, the above bound then holds for any $h \in (0, \mathtt{L}_{\mathtt{N}}^{-1}]$.  This establishes \eqref{convolution lower bound}. 
\end{proof}

\subsection{A uniform method of stationary phase}
Our next task is to study the behaviour of $\widehat{\pi}(\xi)$.  This information is given in Lemma \ref{measure bounds}, below, which we will prove using basic stationary phase techniques.  The following elementary lemma will simplify the argument. 
\begin{lemma}\label{splitting lemma}
Fix any constant $\mathtt{a} \in (0,1]$ and any collection $\{x_1,\ldots,x_n\}$ of nonnegative real numbers, not all zero.  Then there exists $k \in \{1,\ldots, n\}$ such that $x_k \neq 0$ and 
\[ \frac{x_i}{x_k} \leq \mathtt{a}^{-n} \quad \text{if } 1 \leq i \leq k, \quad\quad \frac{x_i}{x_k} \leq \mathtt{a}   \quad \text{if }  k < i \leq n. \]
\end{lemma}

\begin{proof}
We will induct on $n$.  The base case $n = 1$ is trivial.  Assume that $n \geq 2$ and that the lemma holds with $n-1$ in place of $n$.  Applying the induction hypothesis on $\{x_1, \ldots x_{n-1} \}$, let $k_0 \in \{1,\ldots, n-1\}$ be an index such that $x_{k_0} \neq 0$ and $x_i/x_{k_0} \leq \mathtt{a}^{-n+1}$ if $1 \leq i \leq k_0$ and $x_i/x_{k_0} \leq \mathtt{a}$ if $k_0 < i \leq n-1$.  If $x_n/x_{k_0} \leq \mathtt{a}$, then the conclusion of the lemma holds with $k = k_0$.  Assume that $x_n/x_{k_0} \geq \mathtt{a}$.  Then for $i < n$ we have
\begin{align*}
\frac{x_i}{x_n} = \frac{x_i}{x_{k_0}}\cdot\frac{x_{k_0}}{x_n} \leq \mathtt{a}^{-n+1}\mathtt{a}^{-1} = \mathtt{a}^{-n},
\end{align*}
and the conclusion of the lemma holds with $k = n$.
\end{proof}

\begin{lemma}\label{measure bounds}
There exists an admissible constant $\mathtt{L}_{\mathtt N}$ with the following property: For every function $\Phi\colon [0,1] \rightarrow \mathbb R^d$ in standard form that is vanishing of type $\mathtt N$ at the origin, there exists an (inadmissible) index $\mathtt J \geq \mathtt J_0(\Phi)$ depending on $\Phi$ such that  
\begin{align}\label{Fourier decay bound}
\sup \{|\widehat{\pi}(\xi)| \colon \pi = \pi[\Phi; j, c], \, j \geq \mathtt J, c \in (0,1]\} \leq \mathtt{L}_{\mathtt N} (1+|\xi|)^{-{1/\mathtt N}} \quad \text{for all } \xi \in \mathbb R^d.
\end{align}
\end{lemma}

\begin{proof}
Fix $\pi$ of the form $\pi = \pi[\Phi; j, c]$ and $\xi \in \R^d \setminus \{0\}$. It follows from \eqref{measure definition} that  
\begin{equation}  \label{pi hat}
\widehat{\pi}(\xi) = \int e(x \cdot \xi) d\pi(x) = \int_{c}^{1} e(\xi \cdot \Phi^j(s)) ds,  \quad \text{where } e(t) \colonequals e^{-2\pi i t}.
\end{equation} 
The integral representing $\widehat{\pi}(\xi)$ is a scalar oscillatory integral widely studied in harmonic analysis. Our goal is to apply the well-known method of stationary phase for oscillatory integrals (see \cite[Chapter \textrm{VIII}]{Stein-HA}) to arrive at the desired bound \eqref{Fourier decay bound}. It is important to keep track of the implicit constants in this process to ensure uniformity in the parameters $j$ and $c$; we describe the steps below.

Let us choose, in the following order, a small constant $\mathtt a \in (0,1]$ depending on $\mathtt K_{\mathtt N}$, and a large integer $\mathtt J \geq \mathtt J_0$ depending on $\mathtt a$ and $\Phi$, according to the constraints
 \begin{equation} \label{a J}
 d \mathtt K_{\mathtt N} \mathtt a \leq \frac{1}{8}, \qquad d \mathtt{a}^{-d}\|\Phi\|_{C^{\mathtt N}} 2^{-\mathtt J} \leq \frac{1}{8}.
 \end{equation}  
Here, $\|\cdot\|_{C^n}$ refers to the standard norm on the space of $n$-times continuously differentiable functions $f \colon [0,1] \rightarrow \R^d$, namely
\begin{align*}
\| f \|_{C^n} \colonequals \sum_{\ell=0}^n \sup_{t \in [0,1]} |f^{(\ell)}(t)|.
\end{align*}
We assume now that $j \geq \mathtt J$. By Lemma \ref{splitting lemma}, there exists an index $k \in \{1,\ldots, d\}$ depending on $\mathtt a$ and $\xi$ such that 
\begin{equation} \label{def-k}
\xi_{k} \neq 0, \quad\quad \frac{|\xi_i|}{|\xi_{k}|} \leq \mathtt a^{-d} \quad \text{if $1 \leq i \leq k$}, \quad\quad \frac{|\xi_i|}{|\xi_{k}|} \leq \mathtt a \quad \text{if $k < i \leq d$.} 
\end{equation} 
Let us define a function $\varphi = \varphi_{k}$ by the formula 
\begin{align} \label{def-phikj}
\varphi(s) \colonequals \frac{\xi}{\xi_k}\cdot\Phi^j(s) &= \xi_k^{-1} \sum_{i=1}^{d} 2^{\mathtt n_ij} \xi_i \Phi_i(2^{-j}s) \quad \text{for } s \in [c,1],
\end{align}
so that  \eqref{pi hat} reduces to
\begin{align}
\widehat{\pi}(\xi) &= \int_{c}^{1} e(\xi_k \varphi(s))ds. \label{FT rewritten} 
\end{align} 
We claim that the $\mathtt n_k^{\text{th}}$ order derivative of the phase function $\varphi$ in the oscillatory integral \eqref{FT rewritten} is bounded from below by an absolute positive constant, due to our choice of $\mathtt a$ and $\mathtt J$ in \eqref{a J}. To verify this, we first note that
\begin{align*}
 \varphi^{(\mathtt n_k)}(s) &= \xi_k^{-1} \sum_{i=1}^{d} 2^{(\mathtt n_i - \mathtt n_k)j} \xi_i \Phi_{i}^{(\mathtt n_k)}(2^{-j}s) = \mathrm{I} + \mathrm{II} + \mathrm{III}, 
 \end{align*}
 where
 \begin{align*}
 \mathrm{I} \colonequals  \Phi_{k}^{(\mathtt n_{k})}(2^{-j}s),   \quad
\mathrm{II} &\colonequals  \sum_{i=1}^{k-1}\frac{\xi_i}{\xi_k}2^{(\mathtt n_i-\mathtt n_{k})j}\Phi_i^{(\mathtt n_{k})}(2^{-j}s), \quad \mathrm{III} \colonequals  \sum_{i=k+1}^{d}\frac{\xi_i}{\xi_k}2^{(\mathtt n_i-\mathtt n_{k})j}\Phi_i^{(\mathtt n_{k})}(2^{-j}s).
\end{align*}
Here, any empty sum is treated as zero.  Noting that $2^{-j}s \in (0,2^{-\mathtt{J}_0}]$, we use the left inequality in \eqref{c2} to bound \textrm{I} from below, obtaining $|\mathrm{I}| \geq \frac{1}{2}$. The strict monotonicity \eqref{ordering} of the exponents $\mathtt n_i$, the choice \eqref{def-k} of the index $k$, and the condition \eqref{c1} can be used to estimate \textrm{II} and \textrm{III} from above. A combination of these properties yields  
\begin{align*} |\mathrm{II}| &\leq \sum_{i=1}^{k-1} \mathtt a^{-d} 2^{-j} |\Phi_i^{(\mathtt n_k)}(2^{-j}s)| \leq  d\mathtt a^{-d} 2^{-\mathtt J} \|\Phi\|_{C^\mathtt{N}} \leq \frac{1}{8}, \\ 
|\mathrm{III}| &\leq \sum_{i=k+1}^{d} \mathtt a 2^{(\mathtt n_i-\mathtt n_{k})j} \mathtt K_{\mathtt N} (2^{-j}s )^{(\mathtt n_i-\mathtt n_k)} \leq d \mathtt a \mathtt K_{\mathtt N} \leq \frac{1}{8},
\end{align*} 
where the last step in both inequalities follows from \eqref{a J}.  As a result, we obtain
\begin{equation} \label{stationary phase prep} |\varphi^{(\mathtt n_k)}(s)| \geq |\mathrm{I}| - |\mathrm{II}| - |\mathrm{III}| \geq \frac{1}{2} - \frac{1}{8} - \frac{1}{8} = \frac{1}{4} \quad \text{for all } s \in [c, 1].
\end{equation}  
This lower bound is critical to the application of the method of stationary phase, which proceeds via two cases.  

{\em{Case 1:}}  First suppose that $\mathtt n_k \geq 2$. Using van der Corput's lemma (see \cite[Ch.~\textrm{VIII}, \S 1.2, Proposition 2]{Stein-HA}), one can find an absolute (hence admissible) constant $c_{\mathtt n_k} \geq 1$ such that 
\begin{align}\label{case 1 initial est}
|\widehat{\pi}(\xi)| = \Big\vert\int_{c}^1 e^{-2\pi i \xi_k\varphi(s)}ds\Big\vert \leq c_{\mathtt n_k}|\xi_k|^{-1/\mathtt{n}_k}. 
\end{align}
Property \eqref{def-k} implies that
\begin{align*}
|\xi| \leq |\xi_1| + \cdots + |\xi_d| \leq (k\mathtt{a}^{-d} + (d-k)\mathtt a)|\xi_k| \leq d\mathtt{a}^{-d}|\xi_k|.
\end{align*}
Inserting this into \eqref{case 1 initial est} and considering also the trivial estimate $|\widehat{\pi}(\xi)| \leq \|\pi\| \leq 1$, we find that
\begin{align}\label{case1}
|\widehat{\pi}(\xi)| \leq \min\Big\{1, c_{\mathtt{n}_k} \Big(\frac{\mathtt{a}^{d}}{d}|\xi|\Big)^{-1/\mathtt{n}_k}\Big\} &\leq 2c_{\mathtt{n}_k}\Big(\frac{\mathtt{a}^{d}}{d}\Big)^{-1/\mathtt{n}_k}(1+|\xi|)^{-1/\mathtt{n}_k}\\
\notag& \leq 2c_{\mathtt n_k} d \mathtt{a}^{-d}(1+|\xi|)^{-1/\mathtt{N}} \leq \mathtt{L}_{\mathtt N}(1+|\xi|)^{-1/\mathtt{N}}
\end{align}
with $\mathtt{L}_{\mathtt N} = 2d\mathtt{a}^{-d}\max\{c_2, \ldots, c_{\mathtt N}\}$.

{\em{Case 2:}} Next suppose that $\mathtt n_k = 1$. In view of \eqref{ordering}, this means  that $k=1$.  Property \eqref{def-k} thus implies that 
\begin{equation} \label{xi1 large}
 |\xi| \leq |\xi_1| + \cdots + |\xi_d| \leq (1+(d-1)\mathtt a)|\xi_1| \leq d|\xi_1|. 
\end{equation} 
 By virtue of \eqref{def-phikj}, \eqref{c1}, \eqref{c2}, and \eqref{a J}, we have 
 \begin{align} |\varphi''(s)| = \Bigl|\sum_{i=1}^{d} 2^{(\mathtt n_i - 2) j} \frac{\xi_i}{\xi_1} \Phi_i''(2^{-j}s)  \Bigr|  &\leq 2^{-j} \|\Phi\|_{C^\mathtt{N}} + \sum_{i=2}^{d} 2^{(\mathtt n_i - 2) j} \mathtt a \mathtt K_{\mathtt N} (2^{-j}s)^{\mathtt n_i-2} \nonumber \\ 
&\leq 2^{-\mathtt J} \|\Phi\|_{C^\mathtt{N}} + d\mathtt a \mathtt{K}_{\mathtt N} \leq 2d \mathtt K_{\mathtt N} \leq \mathtt{L}_{\mathtt N}  \label{ibp-prep}
\end{align}    
for a suitable choice of $\mathtt{L}_{\mathtt N}$.  Combining the lower bound \eqref{stationary phase prep} on $|\varphi'|$ and the upper bound \eqref{ibp-prep} on $|\varphi''|$ with integration by parts, we obtain 
\begin{align*}
|\widehat{\pi}(\xi)| =  \Bigl| \int_{c}^{1} e(\xi_1 \varphi(s)) ds \Bigr| &= \Bigl| \int_{c}^{1} \frac{1}{2\pi\xi_1 \varphi'(s)} \cdot\frac{d}{ds} [e(\xi_1 \varphi(s)) ] ds \Bigr| \\
&= \frac{1}{2\pi|\xi_1|}\biggl| \frac{e(\xi_1 \varphi(s))}{\varphi'(s)} \Bigr]_{c}^{1} + \int_{c}^{1} \frac{\varphi''(s)}{(\varphi'(s))^2}e(\xi_1 \varphi(s)) ds \biggr| \leq \frac{ \mathtt{L}_{\mathtt N}}{|\xi_1|}
\end{align*} 
for some (larger) choice of $\mathtt{L}_\mathtt{N}$.  We also have the trivial bound $|\widehat{\pi}(\xi)| \leq \|\pi\| \leq 1$.  These estimates, together with \eqref{xi1 large}, imply that
\begin{align} \label{case2}
|\widehat{\pi}(\xi)|  \leq \min \Bigl\{ 1, \frac{d\mathtt{L}_{\mathtt N}}{|\xi|} \Bigr\} \leq \mathtt{L}_{\mathtt N} (1+|\xi|)^{-1} \leq  \mathtt{L}_{\mathtt N} (1+|\xi|)^{-{1/\mathtt N}},
\end{align}
where here we have allowed the value of $\mathtt{L}_{\mathtt N}$ to change between the first two occurences.  Combining the conclusions \eqref{case1} and \eqref{case2} of the two cases completes the proof of \eqref{Fourier decay bound}.
\end{proof}

\subsection{Proof of Proposition \ref{spectral gap lemma}} \label{spectral gap lemma proof} 
We begin by defining the admissible constant $\mathtt{L}_{\mathtt N}$.  Let $\mathtt M_1$ be the admissible constant in Corollary \ref{convolution estimate corollary} (appearing there as $\mathtt{L}_{\mathtt N}$) when applied with
\begin{align}\label{def a_0}
\mathtt{a} = \frac{1}{2}\cdot 15^{-d}|B(0;1)|;
\end{align}
here, $|\cdot|$ refers to Lebesgue measure.  Let $\mathtt M_2$ be the admissible constant in Lemma \ref{measure bounds} (also appearing there as $\mathtt{L}_{\mathtt N}$).  We set
\begin{align*}
\mathtt{L}_{\mathtt N} \colonequals \max\{\mathtt M_1, \mathtt M_2, 2\pi|B(0;1)|+2+2\gamma_{\mathtt N}^{-1}\},
\end{align*}
where $\gamma_\mathtt{N}$ is the constant defined in \eqref{def sigma_N}.

Next we fix a function $\Phi \colon [0,1] \rightarrow \R^d$ in standard form that is vanishing of type $\mathtt{N}$ at the origin, and we define the inadmissible integer $\mathtt J(\Phi)$.  For this, we can just use the integer $\mathtt J$ from Lemma \ref{measure bounds}.

Now, let $\mathtt A, \mathtt B, \mathtt C$ be any choice of constants satisfying \eqref{AB conditions}, and let $\mu$ be any Borel probability measure $\mu$ on $[0,1]^d$ that obeys \eqref{spectral gap condition}.  We need to prove that \eqref{integral lower bound} holds.  With this aim in mind, fix $\pi = \pi[\Phi; j, \mathtt{A}^{-6d}]$ with $j \geq \mathtt J$, and fix $\delta \in (0, \mathtt{A}^{-3d}]$. Write
\begin{align*}
&\int\mu_\delta \ast \pi \, d\mu = \mathbb I_1 + \mathbb I_2,
\intertext{where}
&\mathbb I_1 \colonequals \int \mu_{\mathtt A^{-3d}} \ast \pi \, d\mu,\quad\quad \mathbb I_2 \colonequals \int (\mu_{\delta} - \mu_{\mathtt A^{-3d}}) \ast \pi \, d\mu.
\end{align*}
We claim that 
\begin{equation} \label{I claim} 
\mathbb I_1 \geq \frac{1}{2} \mathtt{L}_{\mathtt N}^{-1} \mathtt A^{-3d} \quad \text{and} \quad |\mathbb I_2| \leq \mathtt{L}_{\mathtt N} \mathtt A^{-4d}. 
\end{equation}  
This, together with our assumption on $\mathtt A$ in \eqref{AB conditions} and $\mathtt{L}_{\mathtt N} \geq 1$, would imply 
\eqref{integral lower bound}.

We start with $\mathbb I_1$.  Using the constant $\mathtt a$ defined in \eqref{def a_0}, let  
\[ \mathcal G =  \{x \in \supp \mu \colon  \eqref{lower bound ball pi} \text{ holds with }x_0 = x \}. \] 
Setting $h = \mathtt A^{-3d}$ in Corollary \ref{convolution estimate corollary} leads to the pointwise lower bound
\begin{align*} 
\mu_{\mathtt A^{-3d}} \ast \pi (x) &\geq \mathtt{L}_{\mathtt N}^{-1}\mathtt A^{-3d} \quad \text{for every } x \in \mathcal G,
\end{align*}
and consequently
\begin{align*}
\mathbb I_1 \geq \int_{\mathcal G} \mu_{\mathtt A^{-3d}} \ast \pi \, d\mu \geq \mathtt{L}_{\mathtt N}^{-1}\mathtt A^{-3d} \mu(\mathcal G). 
\end{align*}     
Thus, in order to obtain the lower bound for $\mathbb I_1$ claimed in \eqref{I claim}, it suffices to show that $\mu(\mathcal G) \geq 1/2$.  For each $x \in \mathcal G^c \colonequals [0,1]^d \setminus \mathcal G$, there exists by definition a ball $B_x$ of radius $r_x \in (0,1]$ centred at $x$ such that 
\begin{equation} \label{def-G}   
\mu(B_x) \leq \mathtt a r_x^d \leq \mathtt c |B_x|, \quad \text{ where } \mathtt c \colonequals \frac{\mathtt a}{|B(0;1)|} = \frac{1}{2}\cdot 15^{-d}. 
\end{equation}
The balls $\{B_x \colon x \in \mathcal G^c \}$ clearly cover $\mathcal G^c$ and lie within $[-1,2]^d$. The Vitali covering lemma, \cite[Lemma 1.9]{Falconer-GFS}, yields a countable set $X \subseteq \mathcal G^c$ such that the subcollection $\{B_x \colon x \in X\}$ continues to cover $\mathcal G^c$, but their scaled counterparts $\{\frac{1}{5}B_x\colon x\in X\}$ are pairwise disjoint. Here, $mB_x$ denotes a ball with the same centre as $B_x$ but $m$ times its radius. This leads to 
\begin{align*}
1 - \mu(\mathcal G) &= \mu([0,1]^d \setminus \mathcal G) \leq \sum_{x \in X}\mu(B_x) \leq \mathtt c \sum_{x \in X}|B_x|\\ &= 5^d\mathtt c \sum_{x \in X}\Big|\frac{1}{5}B_x\Big| =  5^d \mathtt c \Bigl| \bigcup_{x \in X} \frac{1}{5}B_x\Bigr| \leq  5^d \mathtt c |[-1,2]^d| = 15^d\mathtt c = \frac{1}{2}.
\end{align*}
Thus, $\mu(\mathcal G) \geq 1/2$, concluding the estimation of $\mathbb I_1$.

We now turn to $\mathbb I_2$.  By Plancherel's theorem, we have
\begin{align*}
|\mathbb I_2| \leq \int|\widehat{\mu}(\xi)|^2 |\widehat{\psi}(\delta\xi)-\widehat{\psi}({\mathtt A^{-3d}}\xi)|| \widehat{\pi}(\xi)| d\xi.
\end{align*} 
We will estimate this integral by breaking its domain into three pieces:  low frequencies  $\{|\xi| \leq \mathtt A\}$, moderate frequencies $\{\mathtt A \leq |\xi| \leq \mathtt B\}$, and high frequencies $\{|\xi| \geq \mathtt B\}$.  Beginning with the low-frequency piece, we use the trivial bounds $\|\widehat{\mu}\|_\infty \leq 1$ and $\|\widehat{\pi}\|_\infty \leq 1$ as well as  \eqref{conditions-psi} to get
\begin{align*}
\int_{|\xi| \leq \mathtt A} |\widehat{\mu}(\xi)|^2 |\widehat{\psi}(\delta\xi)-\widehat{\psi}({\mathtt A^{-3d}}\xi)|| \widehat{\pi}(\xi)| d\xi &\leq \int_{|\xi| \leq \mathtt A}\pi[(\delta|\xi|)^2 + (\mathtt{A}^{-3d}|\xi|)^2]d \xi\\
&\leq 2\pi|B(0;\mathtt A)|(\mathtt A^{-3d}\mathtt A)^2 \leq 2\pi|B(0;1)|{\mathtt A}^{-4d}.
\end{align*}
We use the spectral gap hypothesis in \eqref{spectral gap condition} to control the moderate-frequency piece, namely
\begin{align*}
\int_{|\xi| \in [\mathtt A, \mathtt B]} |\widehat{\mu}(\xi)|^2 |\widehat{\psi}(\delta\xi)-\widehat{\psi}({\mathtt A^{-3d}}\xi)|| \widehat{\pi}(\xi)| d\xi \leq 2\int_{|\xi| \in [\mathtt A, \mathtt B]}|\widehat{\mu}(\xi)|^2 d\xi \leq 2\mathtt{A}^{-4d};
\end{align*}
here, we have also used that $\|\widehat{\psi}\|_\infty = 1$ (from \eqref{conditions-psi}).  We are left to estimate the high-frequency piece.  For this we use Lemma \ref{measure bounds}, definitions \eqref{def energy} and \eqref{def sigma_N}, the energy condition in \eqref{spectral gap condition}, and our assumption on $\mathtt B$ in \eqref{AB conditions}.  We obtain
\begin{align*}
\int_{|\xi| \geq \mathtt B} |\widehat{\mu}(\xi)|^2 |\widehat{\psi}(\delta\xi)-\widehat{\psi}({\mathtt A^{-3d}}\xi)|| \widehat{\pi}(\xi)| d\xi &\leq 2\mathtt{L}_{\mathtt N} \int_{|\xi| \geq \mathtt B}|\widehat{\mu}(\xi)|^2|\xi|^{-\frac{1}{\mathtt N}}d\xi\\ &\leq 2\mathtt{L}_{\mathtt N} 
\mathtt B^{-\frac{1}{2\mathtt N}}\gamma_{\mathtt N}^{-1} I_{\sigma_{\mathtt N}}(\mu) \leq 2\mathtt{L}_{\mathtt N} \mathtt B^{-\frac{1}{2\mathtt N}}\gamma_{\mathtt N}^{-1}\mathtt C \leq  2\gamma_{\mathtt N}^{-1}\mathtt{A}^{-4d}.
\end{align*}
Now, summing the bounds for the three pieces and recalling our choice of choice of $\mathtt{L}_{\mathtt N}$, it follows that $|\mathbb{I}_2| \leq \mathtt{L}_{\mathtt N} \mathtt{A}^{-4d}$, as claimed.  This completes the proof.  \qed

\section{Constructing a suitable measure:  Proof of Proposition \ref{existence of measure}} \label{measure construction section}

The goal of this section is to prove Proposition \ref{existence of measure}, which establishes the existence of a certain measure. Before embarking on the construction of this measure, we pause to collect a few necessary tools.

\subsection{Measure-theoretic preliminaries}
Let $\vec{n} = (n_1, \ldots, n_d)$ be a vector with positive integer entries, and recall the definition of $\fD^\ast = \fD^\ast[\vec{n}]$ given in \eqref{def of D^*}.  Here, we do not assume that the entries of $\vec{n}$ are distinct or ordered.  The aim of this subsection is to define an anisotropic, dyadic version of ``Hausdorff content" using the collection $\fD^*$ and establish connections between it and the standard notion of Hausdorff dimension.  Although we opt to work from first principles, these connections can also be well understood using the broader framework of Hausdorff dimension in metric spaces; see \cite[Chapter 4]{Mattila}.  This perspective is described in Section \ref{metric space section}.  The lemmas stated in the present subsection may be unsurprising to experts. However, they do not appear in standard textbooks in the form that we need, and thus we provide their proof in the Appendix.

Our version of Hausdorff content is defined as follows:  Let $E \subseteq \R^d$ and $s \geq 0$.  Then
\begin{align*}
\fH^s_{\fD^*}(E) \colonequals \inf\Big\{\sum_{Q \in \fQ} \ell(Q)^s \colon \fQ \subseteq \fD^* \text{ and } E \subseteq \bigcup \fQ\Big\},
\end{align*}
where $\ell(Q)$ is as in \eqref{def ell_Q}.  We are particularly interested in conditions on $E$ and $s$ that guarantee the positivity of $\fH_{\fD^*}^s(E)$.  Let
\begin{align}\label{definition S, N}
\mathtt N = \mathtt N[\vec{n}] \colonequals \max\{n_1,\ldots,n_d\} \quad\quad \text{and} \quad\quad \mathtt S = \mathtt S[\vec{n}] \colonequals n_1 + \ldots + n_d. 
\end{align}
The following lemma addresses this question and establishes the basic connection between $\dim_{\mathrm H} E$ and $\fH_{\fD^*}^s(E)$.

\begin{lemma}\label{positivity}
If $E \subseteq \R^d$ and $0 \leq s < \mathtt S - (d-\dim_{\mathrm H} E)\mathtt N$, then $\fH_{\fD^*}^s(E) > 0$.
\end{lemma}

We will also need a restricted version of $\fH_{\fD^*}^s$, namely
\begin{align*}
\fH_{\fD_J^*}^s(E) \colonequals \inf\Big\{\sum_{Q \in \fQ} \ell(Q)^s \colon \fQ \subseteq \fD_J^* \text{ and } E \subseteq \bigcup \fQ\Big\},
\end{align*}
where $\fD_J^*$ is as in \eqref{def of D_J^*}.  From the definitions, it is clear that $\fH_{\fD^*}^s(E) \leq \fH_{\fD_J^*}^s(E)$.  The next lemma states that, when $E$ is contained in an element of $\fD_{\mathtt J}^*$, this inequality can be reversed.  This will be helpful in locating the dyadic box $\mathtt Q$ referenced in the proposition statement.

\begin{lemma}\label{reversed inequality}
Let $E \subset \R^d$ be a subset of some element of $\fD_J^*$ for some $J$.  Then $\fH_{\fD^*}^s(E) = \fH_{\fD_J^*}^s(E)$ for all $s \geq 0$.
\end{lemma}

Next, we record a version of Frostman's lemma adapted to $\fH_{\fD^*}^s$.  We will use this to construct the measure $\nu$ referenced in the proposition statement.
\begin{lemma}\label{alternate Frostman}
Let $E \subset \R^d$ be a compact set, and let $s \geq 0$.  Then there exists a Borel measure $\vartheta$ supported on $E$ such that $\|\vartheta\| \geq \fH_{\fD^*}^s(E)$ and $\vartheta(Q) \leq \ell(Q)^s$ for every $Q \in \fD^*$. 
\end{lemma}

Our final lemma gives a connection between the Frostman condition for $\fD^*$ and the finiteness of energy integrals.  It will help us verify that the blown-up measure $\mu \colonequals \nu^\mathtt Q$ satisfies the energy condition in \eqref{spectral gap condition}.
\begin{lemma}\label{energy bound}
There exists a decreasing function $\mathtt E \colon (0,\infty) \rightarrow [1,\infty)$, depending only on $d$, with the following property:  For any constant $L \geq 0$ and any exponents $\sigma, s$ with $\sigma \in (0,d)$ and $s > \sigma + \mathtt S - d$, one has
\begin{align*}
\sup\Big\{I_\sigma(\vartheta) \colon \supp\vartheta \subseteq [0,1]^d,~\|\vartheta\| \leq 1, ~\sup_{Q \in \fD^*}\frac{\vartheta(Q)}{\ell(Q)^s} \leq L\Big\} \leq L\mathtt E(s -\sigma-\mathtt S+d).
\end{align*}
\end{lemma}

\subsection{Proof of Proposition \ref{existence of measure}}
We now turn to the proof of Proposition \ref{existence of measure}, beginning with the selection of the constants $\mathtt A, \mathtt B, \mathtt C$ and $\varepsilon$.

\subsubsection{Choice of constants}
We start by setting
\begin{align}\label{choice of C}
\mathtt C \colonequals 4 \mathtt E\Big(\frac{1}{4\mathtt N}\Big),
\end{align}
where $\mathtt E$ is the function from Lemma \ref{energy bound}.  Let $\varphi \colon \R \rightarrow \R$ be a fixed nonnegative bump function supported in $[0,1)^d$ with $\int \varphi = 1$ and $\|\varphi\|_\infty \leq 2$.  This function will be used in the construction of the measure $\nu$ and in the verification that its blow-up $\mu$ obeys the energy and spectral gap conditions in \eqref{spectral gap condition}.  The values of $\mathtt A$ and $\mathtt B$ will depend on $\varphi$.  We define them as follows:  
Let $\mathtt A$ be chosen large enough that
\begin{align}\label{choice of A}
\mathtt{A}^d \geq 4\mathtt L_{\mathtt N}^2 \quad\quad \text{and} \quad\quad \int_{|\xi| \geq \mathtt A}|\widehat{\varphi}(\xi)|d\xi \leq \frac{1}{2}\mathtt{A}^{-4d},
\end{align}
where $\mathtt L_{\mathtt N}$ is the constant appearing in \eqref{AB conditions}.  Let $\mathtt B$ be chosen sufficiently large relative to $\mathtt A, \mathtt C$, and $\mathtt L_{\mathtt N}$, as specified in \eqref{AB conditions}.  This concludes our choice of the constants $\mathtt A, \mathtt B, \mathtt C$.

We are left to define $\varepsilon$.  Toward this end, we introduce another large admissible constant $\mathtt T$.  Specifically, we take $\mathtt T$ to be an integer satisfying
\begin{align}\label{choice of T}
4\pi\sqrt{d}\mathtt{B}|B(0;\mathtt B)|2^{-\mathtt T} \leq \frac{1}{2}\mathtt{A}^{-4d}.
\end{align}
We now define
\begin{align}\label{choice of epsilon}
\varepsilon \colonequals \min\Big\{\frac{\log_2(1+2^{-d\mathtt N\mathtt T -2})}{\mathtt N\mathtt T}, \frac{1}{4\mathtt N^2}\Big\}.
\end{align}

\subsubsection{Locating the dyadic box $\mathtt Q$}
At this point, we fix
\begin{itemize}
\item {a function $\Phi \colon [0,1] \rightarrow \R^d$ in standard form and vanishing of type $\mathtt N$ at the origin, and}
\item {a Borel set $K \subseteq \R^d$ with $\dim_{\mathrm H}K > d-\varepsilon$.}
\end{itemize}

Let $\mathtt J \colonequals \mathtt J(\Phi)$ as in Proposition \ref{spectral gap lemma}, and  let $\vec{\mathtt n} \colonequals (\mathtt n_1,\ldots,\mathtt n_d)$ be the vector of integers associated with $\Phi$ in \eqref{ordering}.  For the remainder of this section, $\vec{\mathtt n}$ should be used whenever an object depends on a vector of integers.  So, for example,  $\fD^* \colonequals \fD^*[\vec{\mathtt n}]$ and $\fD_{J}^* \colonequals \fD_J^*[\vec{\mathtt n}]$.

Our next goal is to locate a box $\mathtt Q \in \fD_{\mathtt J}^*$ such that $K \cap \overline{\mathtt Q}$ will support a measure whose blow-up the obeys energy and spectral gap conditions in \eqref{spectral gap condition}.  Henceforth, we will assume that $K$ is compact.  A straightforward application of Frostman's lemma for Borel sets (e.g.~\cite[Theorem 2.7]{Mattila2}) shows that $K$ contains a compact subset whose Hausdorff dimension strictly exceeds $d-\varepsilon$.  Therefore, this assumption is permissible.

Using Lemma \ref{positivity} and our dimension assumption on $K$, we select some
\begin{align}\label{choice of s}
s \in (\mathtt S -  \mathtt N \varepsilon, \mathtt S]
\end{align}
such that $\fH_{\fD^*}^s(K) > 0$.  It follows that $\fH_{\fD_{\mathtt J}^*}^s(K) > 0$ as well, and we also have $\fH_{\fD_{\mathtt J}^*}^s(K) < \infty$ trivially. We claim for each $\delta > 0$ there exists $Q \in \fD_{\mathtt J}^*$ such that
\begin{align}\label{high density bound}
\fH_{\fD_{\mathtt J}^*}^s(K \cap Q) \geq (1-\delta)\ell(Q)^s.
\end{align}
To see this, note that for each $c > 0$ there exists $\fQ_c \subseteq \fD_{\mathtt J}^*$ such that $\fQ_c$ covers $K$ and
\begin{align*}
\sum_{Q \in \fQ_c} \ell(Q)^s \leq \fH_{\fD_{\mathtt J}^*}^s(K) + c.
\end{align*}
If there were some $\delta > 0$ such that no $Q \in \fD_{\mathtt J}^*$ satisfied \eqref{high density bound}, then we would have
\begin{align*}
\fH_{\fD_{\mathtt J}^*}^s(K) \leq \sum_{Q \in \fQ_c}\fH_{\fD_{\mathtt J}^*}^s(K \cap Q) \leq (1-\delta)\sum_{Q \in \fQ_c}\ell(Q)^s \leq (1-\delta)(\fH_{\fD_{\mathtt J}^*}^s(K) + c),
\end{align*}
and taking $c$ sufficiently small would produce a contradiction.  This proves the claim.  By Lemma \ref{positivity}, condition \eqref{high density bound} is equivalent to the statement that
\begin{align}\label{high density bound D^*}
\fH_{\fD^*}^s(K \cap Q) \geq (1-\delta)\ell(Q)^s.
\end{align}
We now define $\mathtt Q$ to be any $Q \in \fD_{\mathtt J}^*$ such that \eqref{high density bound D^*} holds with $\delta \colonequals 2^{-\mathtt S \mathtt T-2}$.

\subsubsection{Constructing the measure $\nu$}
Let $\ch(\mathtt Q)$ denote the set of $\mathtt T^{\mathrm{th}}$-generation descendants of $\mathtt Q$ in $\fD^*$; that is,
\begin{align*}
\ch(\mathtt Q) \colonequals \{q \in \fD^* \colon q \subseteq \mathtt Q,~\ell(q) = 2^{-\mathtt T}\ell(\mathtt Q)\}.
\end{align*}
For the sake of readability, we will from now on denote elements of $\ch(\mathtt Q)$ using the lowercase letter $q$, while the capital letter $Q$ will continue to refer to generic elements of $\fD^*$.  We claim that
\begin{align}\label{moderate density bound}
\fH_{\fD^*}^s(K \cap q) \geq \frac{1}{2}\ell(q)^s \quad \text{for every }  q \in \ch(\mathtt Q).
\end{align}
To see this, let
\begin{align*}
\fG \colonequals \Big\{q \in \ch(\mathtt Q) \colon \fH_{\fD^*}^s(K \cap q) \geq \frac{1}{2}\ell(q)^s\Big\}
\end{align*}
and suppose for contradiction that $\fG \subsetneq \ch(\mathtt Q)$.  Then, by our choice of $\mathtt Q$, the inequality $\fH_{\fD^*}^s(K \cap q) \leq \ell(q)^s$, our choice of $s$ in \eqref{choice of s}, and the definition of $\varepsilon$ in \eqref{choice of epsilon}, we have
\begin{align*}
1-2^{-\mathtt S\mathtt T-2} &\leq \frac{\fH_{\fD^*}^s(K \cap \mathtt Q)}{\ell(\mathtt Q)^s}\\
&\leq \sum_{q \in \fG}\frac{\ell(q)^s}{\ell(\mathtt Q)^s} + \frac{1}{2}\sum_{q \in \ch(\mathtt Q) \setminus \fG}\frac{\ell(q)^s}{\ell(\mathtt Q)^s}\\
&= \sum_{q \in \ch(\mathtt Q)}\frac{\ell(q)^s}{\ell(\mathtt Q)^s} - \frac{1}{2}\sum_{q \in \ch(\mathtt Q)\setminus \fG}\frac{\ell(q)^s}{\ell(\mathtt Q)^s}\\
&\leq 2^{\mathtt S \mathtt T}2^{-\mathtt T s} - \frac{1}{2}\cdot 2^{-\mathtt T s} < 2^{\varepsilon \mathtt N \mathtt T} - 2^{-\mathtt S \mathtt T -1} \leq 1 + 2^{-\mathtt S \mathtt T-2}-2^{-\mathtt S \mathtt T-1}.
\end{align*}
This gives a contradiction (note the strict inequality), and so the claim is proved.

Combining the lower bound \eqref{moderate density bound} with Lemma \ref{alternate Frostman}, we get for each $q \in \ch(\mathtt Q)$ a Borel measure $\vartheta_q$ supported on $K \cap \overline{q}$ such that
\begin{align}\label{mass lower bound}
\|\vartheta_q\| \geq \frac{1}{2}\ell(q)^s \quad\quad\text{and}\quad\quad \vartheta_q(Q) \leq \ell(Q)^s \quad\text{for every } Q \in \fD^*.
\end{align}
The measure $\nu$ will be defined as a weighted sum of the measures $\vartheta_q$.  Specifically, let $\varphi$ be the bump function introduced above, and let ${\bf T}_{\mathtt Q}$ be the rescaling map defined in \eqref{rescaling map} that takes $\mathtt Q$ to $[0,1)^d$.  Let
\begin{align*}
w(q) \colonequals \int_{{\bf T}_{\mathtt Q}(q)} \varphi \quad\quad\text{and}\quad\quad \overline{\vartheta}_q \colonequals \frac{w(q)\ell(\mathtt Q)^s}{\|\vartheta_q\|}\vartheta_q
\end{align*}
for each $q \in \ch(\mathtt Q)$.  With these definitions in place, we take $\nu$ to be
\begin{align*}
\nu \colonequals \sum_{q \in \ch(\mathtt Q)} \overline{\vartheta}_q.
\end{align*}
It is clear that $\nu$ is supported on $K \cap \overline{\mathtt Q}$ and has total mass
\begin{align}\label{nu mass}
\|\nu\| = \sum_{q \in \ch(\mathtt Q)}\|\overline{\vartheta}_q\| = \sum_{q \in \ch(\mathtt Q)}w(q)\ell(\mathtt Q)^s = \ell(\mathtt Q)^s \int \varphi = \ell(\mathtt Q)^s > 0.
\end{align}

Our next goal is to analyze the blow-up of $\nu$ with respect to $\mathtt Q$.  For this measure to be well defined, it must satisfy $\nu(\mathtt Q) > 0$.  In view of \eqref{nu mass}, this could fail only if $\supp \nu$ lies entirely within the boundary of $\mathtt Q$.  We claim that
\begin{align}\label{zero on boundary}
\vartheta_q(\partial q) = 0 \quad \text{for every } q \in \ch(\mathtt Q);
\end{align}
here, $\partial E$ denotes the boundary of $E$.  This claim implies that $\nu(\partial \mathtt Q) = 0$, thus confirming that $\nu(\mathtt Q) > 0$, and says that for all intents and purposes the measures $\vartheta_q$ (and $\overline{\vartheta}_q$) have pairwise disjoint supports.  To prove \eqref{zero on boundary}, fix $q \in \ch(\mathtt Q)$ and let $J$ be such that $q \in \fD_J$.  For each $j \geq J$, let
\begin{align*}
\fQ_j(q) \colonequals \{Q \in \fD_j \colon Q \cap q \neq \emptyset\}.
\end{align*}
By the Frostman-type condition in \eqref{mass lower bound}, we have
\begin{align}\label{boundary of q union bound}
\vartheta_q(\partial q) \leq \#\fQ_j(q)2^{-js}.
\end{align}
We can estimate $\#\fQ_j(q)$ as follows:  For each $i \in \{1,\ldots, d\}$, the boundary of $q$ contains two $(d-1)$-dimensional faces with side-lengths $2^{-n_1J}, \ldots, [2^{-n_iJ}],\ldots, 2^{-n_dJ}$, where the term in brackets is omitted.  Together, these account for all of the faces of $\partial q$.  The two faces corresponding to $i$ each intersect exactly
\begin{align*}
\frac{\prod_{k \neq i} 2^{-n_kJ}}{\prod_{k \neq i} 2^{-n_kj}} = 2^{(\mathtt S - n_i)(j-J)}
\end{align*}
boxes in $\fD_j$.  Hence,
\begin{align}\label{Q_j cardinality bound}
\#\fQ_j(q) \leq 2\sum_{i=1}^d  2^{(\mathtt S - n_i)(j-J)} \leq 2d 2^{(\mathtt S - 1)(j-J)}.
\end{align}
Combining \eqref{boundary of q union bound} and \eqref{Q_j cardinality bound}, we get
\begin{align*}
\vartheta_q(\partial q) \leq 2d2^{-(\mathtt S - 1)J}2^{j(\mathtt S-1-s)}.
\end{align*}
By our choice of $s$ in \eqref{choice of s} and $\varepsilon$ in \eqref{choice of epsilon} (specifically, that $\mathtt N\varepsilon \leq 1$), we have $\mathtt S - 1 - s < 0$.  Thus, sending $j \rightarrow \infty$ yields \eqref{zero on boundary}.

\subsubsection{Verification of the energy condition}
It remains to show that the blow-up $\mu \colonequals \nu^\mathtt{Q}$ satisfies the energy and spectral gap conditions in \eqref{spectral gap condition}.  We begin with the former, namely that $I_{\sigma_{\mathtt N}}(\mu) \leq \mathtt C$, with $\sigma_\mathtt{N}$ and $\mathtt C$ as defined in \eqref{def sigma_N} and \eqref{choice of C}, respectively.  By our choice of $s$ in \eqref{choice of s} and $\varepsilon$ in \eqref{choice of epsilon}, we have
\begin{align*}
s - \sigma_\mathtt{N} - \mathtt S + d \geq -\mathtt N \varepsilon + \frac{1}{2\mathtt N} \geq -\frac{1}{4\mathtt N} + \frac{1}{2\mathtt N} = \frac{1}{4\mathtt N}.
\end{align*}
Therefore, by Lemma \ref{energy bound}, it suffices to show that $\mu$ obeys the Frostman condition
\begin{align}\label{dyadic Frostman condition}
\mu(Q) \leq  4\ell(Q)^s \quad \text{for every } Q \in \fD^*.
\end{align}  
For each $q \in \ch(\mathtt Q)$, we use \eqref{choice of s} and the first property in \eqref{mass lower bound} to get that
\begin{align}\label{size of weight}
w(q)\ell(\mathtt Q)^s \leq \|\varphi\|_\infty|{\bf T}_{\mathtt Q}(q)|\ell(\mathtt Q)^s &\leq 2\Big(\frac{\ell(q)}{\ell(\mathtt Q)}\Big)^\mathtt{S}\ell(\mathtt Q)^s\\
\notag &= 2\Big(\frac{\ell(q)}{\ell(\mathtt Q)}\Big)^{\mathtt S - s}\ell(q)^s \leq 2\ell(q)^s \leq 4\|\vartheta_q\|.
\end{align}
Consequently, by the second property in \eqref{mass lower bound}, each $\overline{\vartheta}_q$ obeys
\begin{align*}
\overline{\vartheta}_q(Q) \leq 4 \vartheta_q(Q) \leq 4 \ell(Q)^s \quad \text{for every } Q \in \mathtt \fD^*.
\end{align*}
We claim that this property implies condition \eqref{dyadic Frostman condition} with $\nu$ in place of $\mu$, i.e.
\begin{align}\label{nu Frostman condition}
\nu(Q) \leq 4\ell(Q)^s \quad \text{for every } Q \in \fD^*.
\end{align}
To see this, fix $Q \in \fD^*$ and consider three cases:
\begin{itemize}
\item{If $\ell(Q) \leq 2^{-\mathtt T}\ell(\mathtt Q)$, then $Q$ intersects at most one box in $\ch(\mathtt Q)$; thus
\begin{align*}
\nu(Q) \leq \max_{q \in \ch(\mathtt Q)} \overline{\vartheta}_q(Q) \leq 4\ell(Q)^s.
\end{align*}}
\item{If $\ell(Q) \geq \ell(\mathtt Q)$, then \eqref{nu mass} implies that
\begin{align*}
\nu(Q) \leq \|\nu\| =  \ell(\mathtt Q)^s \leq \ell(Q)^s.
\end{align*}}
\item{If $2^{-\mathtt T}\ell(\mathtt Q) \leq \ell(Q) \leq \ell(\mathtt Q)$, then using the first line of \eqref{size of weight}, as well as \eqref{choice of s}, we get
\begin{align*}
\nu(Q) &\leq \#\{q \in \ch(\mathtt Q) \colon q \cap Q \neq \emptyset\} \max_{q \in \ch(\mathtt Q)}\|\overline{\vartheta}_q\|\\
&= \Big(\frac{\ell(Q)}{2^{-\mathtt T}\ell(\mathtt Q)}\Big)^\mathtt{S}\max_{q \in \ch(\mathtt Q)}w(q)\ell(\mathtt Q)^s\\
&\leq \Big(\frac{\ell(Q)}{2^{-\mathtt T}\ell(\mathtt Q)}\Big)^\mathtt{S}2^{-\mathtt S \mathtt T+1}\ell(\mathtt Q)^s = 2\Big(\frac{\ell(Q)}{\ell(\mathtt Q)}\Big)^{\mathtt S - s}\ell(Q)^s \leq 2\ell(Q)^s.
\end{align*}}
\end{itemize}
Collectively, these imply \eqref{nu Frostman condition}.  Now, in order to verify \eqref{dyadic Frostman condition}, we need one more simple fact regarding dyadic boxes, namely
\begin{align*}
&{\bf T}_{Q}(Q') \in \fD^*\quad \text{with} \quad \ell({\bf T}_{Q}(Q')) = \frac{\ell(Q')}{\ell(Q)} 
\intertext{and (equivalently)}
&{\bf T}_{Q}^{-1}(Q') \in \fD^*\quad \text{with} \quad \ell({\bf T}_{Q}^{-1}(Q')) = \ell(Q)\ell(Q')  \quad\quad \text{for all } Q,Q' \in \fD^*.
\end{align*}
Combining this observation with the definition of blow-up in \eqref{def blow-up} and \eqref{push-forward measure}, as well as with \eqref{nu mass} and \eqref{nu Frostman condition}, we obtain
\begin{align*}
\mu(Q) = \frac{\nu({\bf T}_{\mathtt Q}^{-1}(Q))}{\|\nu\|} \leq \ell(\mathtt Q)^{-s}4\ell({\bf T}_{\mathtt Q}^{-1}(Q))^s = 4\ell(Q)^s
\end{align*}
for every $Q \in \fD^*$, confirming \eqref{dyadic Frostman condition}.

\subsubsection{Verification of the spectral gap condition}
Finally, we need to prove that $\mu$ satisfies the spectral gap condition in \eqref{spectral gap condition}.  Let $\fD_{\mathtt T, 0}$ denote the set of $\mathtt{T}^{\mathrm{th}}$-generation descendants of $[0,1)^d$ in $\fD^*$; that is,
\begin{align*}
\fD_{\mathtt T, 0} \colonequals \{Q \in \fD_{\mathtt T} \colon Q \subseteq [0,1)^d\} = \{{\bf T}_{\mathtt Q}(q) \colon q \in \ch(\mathtt Q)\}.
\end{align*}
We have
\begin{align}\label{equality on cubes}
\mu(Q) = \frac{\nu({\bf T}_{\mathtt Q}^{-1}(Q))}{\|\nu\|} = \ell(\mathtt Q)^{-s}\|\overline{\vartheta}_{{\bf T}_{\mathtt Q}^{-1}(Q)}\| = w({\bf T}_{\mathtt Q}^{-1}(Q)) = \int_Q \varphi \quad \text{for every } Q \in \fD_{\mathtt T, 0}.
\end{align}
We will treat $\varphi$ as a measure via the formula
\begin{align*}
\int fd\varphi \colonequals \int f\varphi;
\end{align*}
thus \eqref{equality on cubes} becomes $\mu(Q) = \varphi(Q)$.  Let $c_Q$ denote the centre of the box $Q$.  If $Q \in \fD_{\mathtt T,0}$, then
\begin{align*}
\int_Q e^{-2\pi i c_Q \cdot \xi} d\mu(x) = \int_Q e^{-2\pi i c_Q \cdot \xi} d \varphi(x)
\end{align*}
for any $\xi \in \R^d$.  For fixed $\xi$, the function $x \mapsto e^{-2\pi i x \cdot \xi}$ is Lipschitz with constant at most $2\pi|\xi|$.  Since $|x- c_Q| \leq \sqrt{d}\ell(Q)$ for $x \in Q$, it follows that
\begin{align*}
|\widehat{\mu}(\xi) - \widehat{\varphi}(\xi)| &= \Big\vert\int e^{-2\pi ix\cdot \xi}d\mu(x) - \int e^{-2\pi i x \cdot \xi} d\varphi(x)\Big\vert \\
&\leq \sum_{Q \in \fD_{\mathtt T,0}}\Big\vert\int_Q e^{-2\pi ix\cdot \xi}d\mu(x) - \int_Q e^{-2\pi i x \cdot \xi} d\varphi(x)\Big\vert\\
&\leq \sum_{Q \in \fD_{\mathtt T, 0}}\Big(\int_Q |e^{-2\pi ix\cdot \xi} - e^{-2\pi i c_Q \cdot \xi}| d\mu(x) + \int_Q |e^{-2\pi ix\cdot \xi} - e^{-2\pi i c_Q \cdot \xi}| d\varphi(x)\Big)\\
&\leq \sum_{Q \in \fD_{\mathtt T, 0}} 2\pi|\xi|\sqrt{d}2^{-\mathtt T}(\mu(Q)+\varphi(Q)) = 4\pi\sqrt{d}|\xi|2^{-\mathtt T}.
\end{align*}
Now, using that $\|\widehat{\mu}\| \leq 1$, as well as our assumptions on $\mathtt A$ and $\mathtt T$ in \eqref{choice of A} and \eqref{choice of T}, we obtain
\begin{align*}
\int_{|\xi| \in [\mathtt A, \mathtt B]} |\widehat{\mu}(\xi)|^2 d\xi \leq \int_{|\xi| \in [\mathtt A, \mathtt B]} |\widehat{\mu}(\xi)| d\xi &\leq \int_{|\xi| \in [\mathtt A, \mathtt B]} |\widehat{\mu}(\xi)-\widehat{\varphi}(\xi)| d\xi + \int_{|\xi| \in [\mathtt A, \mathtt B]} |\widehat{\varphi}(\xi)| d\xi\\
&\leq 4\pi\sqrt{d}2^{-\mathtt T}\int_{|\xi| \in [\mathtt A, \mathtt B]} |\xi| d\xi + \frac{1}{2}\mathtt{A}^{-4d}\\
&\leq 4\pi\sqrt{d}2^{-\mathtt T}\mathtt B |B(0;\mathtt B)| + \frac{1}{2}\mathtt{A}^{-4d} \leq \mathtt{A}^{-4d},
\end{align*}
which completes the proof. \qed

\section{Anisotropic boxes and Hausdorff dimension in metric spaces}\label{metric space section}
In the previous section, we defined the ``Hausdorff content" $\fH_{\fD^*}^s$ associated to a collection $\fD^*$ of anisotropic dyadic boxes in $\R^d$.  It was used in the following way:  Starting with a set $K \subseteq \R^d$ of very high Hausdorff dimension, we located a box $\mathtt Q \in \fD^*$ such that $K$ had nontrivial (anisotropic) Hausdorff content within each descendant of $\mathtt Q$ of a certain generation.  This enabled us to use a version of Frostman's lemma to construct a measure on $K \cap \overline{\mathtt Q}$ whose blow-up satisfied both conditions in \eqref{spectral gap condition}.  In this section, we reinterpret $\fD^*$ and $\fH_{\fD^*}^s$ in terms of Hausdorff dimension in metric spaces; see \cite[Chapter 4]{Mattila}.  This connection was first stated in \cite{KOS}, though in less generality.

Let $\vec{n} = (n_1,\ldots, n_d)$ be a fixed vector of positive integers, and let $\rho = \rho[{\vec n}]$ be the metric on $\R^d$ given by
\begin{align*}
\rho(x,y) \colonequals \max_{1 \leq i \leq d} (2|x_i-y_i|)^{1/n_i}.
\end{align*}
The closed $\rho$-ball of radius $r$ centred at $x$ takes the form
\begin{align*}
B_\rho(x; r) \colonequals \{y \colon \rho(x,y) \leq r\} = x+\prod_{i=1}^d \Big[-\frac{r^{n_i}}{2}, \frac{r^{n_i}}{2}\Big].
\end{align*}
Elements of $\fD^*[\vec{n}]$ can therefore be viewed as $\rho$-balls.  More precisely, if $Q \in \fD^*$ and $c_Q$ is the centre of $Q$, then
\begin{align*}
\overline{Q} = c_Q + \prod_{i=1}^d \Big[-\frac{\ell(Q)^{n_i}}{2}, \frac{\ell(Q)^{n_i}}{2}\Big] = B_\rho(c_Q, \ell(Q)).
\end{align*}
The $s$-dimensional Hausdorff measure $\mathscr{H}^s = \mathscr{H}^s[\rho]$ is defined for the metric space $(\R^d, \rho)$ in the usual way, namely
\begin{align*}
\mathscr{H}^s(E) \colonequals \lim_{\delta \searrow 0}\mathscr{H}_\delta^s(E),
\end{align*}
where
\begin{align*}
\mathscr{H}_\delta^s(E) \colonequals \inf\Big\{\sum_{U \in \mathscr{U}} \diam(U)^s \colon \mathscr{U} \text{ is a countable cover of $E$, } \sup_{U \in \mathscr{U}}\diam U \leq \delta\Big\}.
\end{align*}
Of course, here $\diam U$ refers to the $\rho$-diameter of $U$.  If $U$ is a $\rho$-ball, then this is $2^{1/\min_i n_i}$ times its radius.  The quantities $\fH_{\fD^*}^s$ and $\fH_{\fD_J^*}^s$ defined in the previous section can be viewed as discrete versions of $\mathscr{H}_\infty^s$ and $\mathscr{H}_{2^{-J}}^s$, respectively.

Lemma \ref{positivity} provided a connection between $\fH_{\fD^*}^s$ and the standard (Euclidean) Hausdorff dimension; namely, if $s$ is not too large relative to $\dim_{\mathrm H} E$, then $\fH_{\fD^*}^s(E) > 0$.  We get a more complete picture by considering Hausdorff dimension in $(\R^d, \rho)$.  This is again defined in the usual way:
\begin{align*}
\dim_{\mathrm{H}^*}E \colonequals \inf\{s \colon \mathscr{H}^s(E) = 0\} = \sup\{s \colon \mathscr{H}^s(E) = \infty\}.
\end{align*}
If $\mathtt N$ and $\mathtt S$ are defined as in \eqref{definition S, N}, then we have the relations
\begin{align}\label{H dim comparison}
\mathtt S - (d-\dim_{\mathrm H} E)\mathtt N \leq \dim_{\mathrm{H}^*} E \leq \mathtt S
\end{align}
for all sets $E \subseteq \R^d$. The proof of \eqref{H dim comparison} is similar to that of Lemma \ref{positivity} (appearing in the Appendix); in particular, the first inequality follows by establishing that $\mathscr{H}_\infty^s(E) > 0$ for all $ s < \mathtt S - (d- \dim_{\mathrm H} E) \mathtt N$.  It is worth noting that when $\vec{n} = (1,1,\ldots,1)$, the corresponding metric $\rho$ is a multiple of the Euclidean sup norm, and the balls it defines are just Euclidean cubes.  Consequently, the definition of Hausdorff dimension in $(\R^d, \rho)$ agrees with the usual one, and \eqref{H dim comparison} reduces to the statement that $\dim_{\mathrm{H}}E \leq d$.

\section{Appendix} \label{appendix}
\subsection{Standardization of a function of finite type} 
\begin{proof}[Proof of Lemma \ref{linear transformation}]
Let $\Theta \colon \mathtt I \rightarrow \R^d$ be a smooth function that is vanishing of type $\mathtt N$ at the origin.  This means that $\mathtt N$ is the smallest integer with the following property:  
\begin{equation} \label{finite type for Theta} 
\text{For every } u \in \mathbb R^d \setminus \{0\} \text{ there exists } n \in \{1, \ldots, \mathtt N\} \text{ such that } u \cdot \Theta^{(n)}(0) \ne 0. 
\end{equation} 
Since $\Theta(0) = 0$, there exists a unique $d$-tuple $\vec{\mathtt m} = (\mathtt m_1,\ldots, \mathtt m_d)$ of positive integers such that 
\begin{align}\label{form of G}
\Theta  = (\Theta_1, \ldots, \Theta_d) \quad \text{with} \quad \Theta_i(t) = t^{\mathtt m_i} \theta_i(t)
\end{align}
for some smooth functions $\theta_i \colon \mathtt I \rightarrow \R^d$ obeying $\theta_i(0) \neq 0$. 
We will refer to $\vec{\mathtt m}$ as the \emph{vanishing pattern} of $\Theta$. Our goal is to produce an invertible linear map ${\bf L} \colon \R^d \rightarrow \R^d$ such that the vanishing pattern $\vec{\mathtt n}$ of ${\bf L} \circ \Theta$ satisfies $\mathtt n_1 < \mathtt n_2 < \cdots < \mathtt n_d = \mathtt N$.  Once we have this, we can apply a diagonal map to ${\bf L} \circ \Theta$ to achieve the condition $\phi_1(0) = \cdots = \phi_d(0) = 1$; such a map leaves the vanishing pattern unchanged. This would establish both requirements \eqref{ordering} and \eqref{coefficients}. Definition \eqref{finite type for Theta} ensures that the type of a function is preserved under invertible linear transformations; i.e., ${\bf L} \circ \Theta$ remains vanishing of type $\mathtt N$ at the origin for any invertible linear transformation ${\bf L}$. In view of this, we will outline a sequence of such transformations, relabelling the transformed function $\Theta$ at each stage, until we reach a function $\Theta$ that obeys  \eqref{standard form}, \eqref{ordering}, and \eqref{coefficients}. The map ${\bf L}$ in the statement of Lemma \ref{linear transformation} is the composition of the maps used to reach this final $\Theta$.

Our first task is to create a map ${\bf L}$ such that $\mathtt N$ is present in the vanishing pattern of ${\bf L} \circ \Theta$.  From \eqref{finite type for Theta}, we know that there exists a nonzero vector $\mathtt u = (\mathtt u_1,\ldots, \mathtt u_d)$ in $\R^d$ such that 
\[ \mathtt u \cdot \Theta^{(n)}(0) = 0 \quad \text{for $1 \leq n \leq \mathtt N-1$} \quad\quad \text{and} \quad\quad  \mathtt u \cdot \Theta^{(\mathtt N)}(0) \ne 0. \] 
We may assume without loss of generality that $\mathtt u_d \neq 0$.  Let ${\bf L}$ be the linear map given by
\[{\bf L}(x) \colonequals (x_1, \ldots, x_{d-1}, \mathtt u \cdot x) \quad \text{ for } x = (x_1, \ldots, x_d) \in \mathbb R^d. \]  
Then ${\bf L} \circ \Theta =:  (\widetilde{\Theta}_1, \ldots, \widetilde{\Theta}_d)$ obeys  
\[ \widetilde{\Theta}_i(t) = \begin{cases} \Theta_i(t) = t^{\mathtt m_i}{\theta}_i(t) &\text{ if } 1 \leq i < d, \\ \mathtt u \cdot \Theta(t) = t^{\mathtt N} \widetilde{\theta}_d(t) &\text{ if } i = d   \end{cases} \]
for some smooth function $\widetilde{\theta}_d$ such that $\widetilde{\theta}_d(0) \neq 0$. 
Thus, $\mathtt N$ appears in the vanishing pattern of ${\bf L} \circ \Theta$ and is necessarily its largest entry.  We relabel ${\bf L} \circ \Theta$ as $\Theta$ and assume it takes the form \eqref{form of G}, with $\max_i \mathtt m_i = \mathtt N$.  The value of $\max_i \mathtt m_i$ will never decrease in the sequence of transformations applied hereafter.

It remains to ensure that the entries of the vanishing pattern of $\Theta$ can be made to obey the strict monotonicity in \eqref{ordering} after a further linear transformation ${\bf L}$. If these entries are already distinct, then the construction of ${\bf L}$ is easy; we can take ${\bf L}$ to be a suitable permutation map.  However, in general the vanishing pattern may have repeated entries, and such coincidences have to be eliminated. With that goal in mind, we claim the following:

\emph{Let $\Theta$ be a function with vanishing pattern $\vec{\mathtt m}$.  If there exist indices $i_0, i_1 \in \{1,\ldots,d\}$ with $i_0 < i_1$ such that $\mathtt m_{i_0} = \mathtt m_{i_1}$, then there exists an invertible linear map ${\bf L}$ such that ${\bf L} \circ \Theta$ has vanishing pattern $\vec{\mathtt m} + \ell {\mathtt e}_{i_1}$ for some $\ell \geq 1$.}

(Here, $\mathtt{e}_1, \ldots, \mathtt{e}_d$ denote the canonical basis vectors of $\mathbb R^d$.) In other words, if the vanishing pattern of $\Theta$ has a repeated value, then there exists a linear map ${\bf L}$ such that the vanishing pattern of ${\bf L}\circ \Theta$ is the same, except for one of the repeated entries having increased.  For now, assume this claim holds.  By applying the claim iteratively, we can remove any coincidences in the vanishing pattern of $\Theta$, as follows. First, set $i_0 = 1$. For every $j > 1$, we check whether $\mathtt m_1 = \mathtt m_j$. If not, then no action is needed. If the equality holds, we apply the claim with  $i_1 = j$ (and $i_0=1$) to get a function whose vanishing pattern has distinct values in the first and the $j^{\text{th}}$ entries; the other entries have not changed.  At the end of $d-1$ such checks with indices $j = 2, 3, \ldots, d$, we obtain a function with the property that the first entry $\mathtt m_1$ of its vanishing pattern is never repeated among the later entries.  This concludes the first step. Next, we repeat this process with $i_0=2$ and $j > 2$.  This produces a function with a vanishing pattern in which neither of the first two entries appears among the later ones.  Proceeding in this way, we reach after $d-1$ steps a function whose vanishing pattern has no repeated values. Finally, we apply a permutation map to rearrange the pattern into increasing order, as suggested above.  Thus, we have constructed ${\bf L}$ as the composition of the linear maps behind our applications of the claim, together with the final permutation.

We are left to prove the claim.  Fix the indices $i_0 < i_1$ with $\mathtt m_{i_0} = \mathtt m_{i_1}$, as specified there. Assuming $\Theta$ takes the form \eqref{form of G}, we define the linear map ${\bf L}$ as follows: 
\begin{align*}
{\bf L}(x) \colonequals y, \quad \text{where} \quad
y_i \colonequals \begin{cases} x_i &\text{if } i \neq i_1,\\
x_{i} - \frac{\theta_{i_1}(0)}{\theta_{i_0}(0)}x_{i_0} &\text{if } i = i_1.
\end{cases}
\end{align*}
The transformed function ${\bf L} \circ \Theta =: (\widetilde{\Theta}_1, \ldots, \widetilde{\Theta}_d)$ satisfies $\widetilde{\Theta}_i(t) = t^{\mathtt m_i}\widetilde{\theta}_i(t)$ for
\begin{align*}
\widetilde{\theta}_i \colonequals \begin{cases}
\theta_i &\text{if } i \neq i_1, \\
\theta_{i_1} - \frac{\theta_{i_1}(0)}{\theta_{i_0}(0)} \theta_{i_0} &\text{if } i = i_1.
\end{cases}
\end{align*}
Now, on one hand, $\widetilde{\theta}_{i_1}(0) = 0$.  On the other hand, $t^{\mathtt m_{i_1}}\widetilde{\theta}_{i_1}(t) = \widetilde{\Theta}_{i_1}(t) = \mathtt u\cdot \Theta(t)$ for some nonzero vector $\mathtt u \in \R^d$, and the finite type hypothesis \eqref{finite type for Theta} implies the existence of an integer $\mathtt n \in \{1,\ldots, \mathtt N\}$ such that  $\mathtt u\cdot \Theta(t) = t^{\mathtt n} \theta_{\mathtt u}(t)$ for some smooth function $\theta_{\mathtt u}$ with $\theta_{\mathtt u}(0) \neq 0$.  Equating $t^{\mathtt m_{i_1}}\widetilde{\theta}_{i_1}(t)$ and $t^{\mathtt n} \theta_{\mathtt u}(t)$ and differentiating $\mathtt n$ times at the origin shows that $\mathtt n > \mathtt m_{i_1}$.  Setting $\ell \colonequals \mathtt n-\mathtt m_{i_1}$, it follows that ${\bf L} \circ \Theta$ has vanishing pattern $\vec{\mathtt m} + \ell {\mathtt e}_{i_1}$, and so the claim is proved.  
\end{proof}

\subsection{Partial-avoidance of graph-like curves}

\begin{proof}[Proof of Proposition \ref{partial-avoidance graphs}]
The proof is quite similar to that of Proposition \ref{generalized-graph-lemma}, so we will omit some details.  Fix ${\bf L}$ and $\Phi \colon \mathtt I \rightarrow \R^d$ satisfying conditions (i)--(iii) in the definition of graph-like curve, with $\mathtt m$ in condition (ii) being the subtype of $\Gamma$.  We may assume without loss of generality that ${\bf L}$ is the identity.  Since $\Gamma$ contains the origin, we necessarily have $0 \in \mathtt I$ and $\Phi(0)= 0$.  Additionally, since $\Gamma$ is of type $\mathtt N$ at the origin, it follows that $\Phi$ is of type at least $\mathtt N$ at zero.  Therefore, there exists a unit vector $\mathtt u \in \R^d$ such that $\mathtt u \cdot \Phi^{(n)}(0) = 0$ for every $n \in \{0,1,\ldots, \mathtt N - 1\}$.  Let $z$ be defined as in the proof of Proposition \ref{generalized-graph-lemma}.  Property  \eqref{z_1 upper bound} still holds with this $z$, but now only for $n \in \{0,1,\ldots, \mathtt N\}$.  Properties \eqref{z derivative condition} and \eqref{z bar lower bound} remain valid without any changes.

Let
\begin{align*}
\overline{s} \colonequals \min\Big\{\frac{\mathtt N}{\mathtt m}, d-\frac{(d-1)\mathtt m}{\mathtt N}\Big\}.
\end{align*}
Lemma \ref{linear transformation} and the definition of subtype given in Subsection \ref{quantitative partial-avoidance} imply that $\mathtt N \geq \mathtt m + d-1$. In particular, we have $\mathtt N > \mathtt m$, so that $\overline{s} \in (1,d)$.  It suffices to prove the proposition for all $s$ sufficiently close to $\overline{s}$.  In view of this, let us fix $s \in [1,\overline{s})$.

Our next goal is to define $K$ as in the proof of Proposition \ref{generalized-graph-lemma}; see the four bullet points in that proof.  There, we chose an arbitrary H\"older continuity exponent $\alpha$ satisfying \eqref{Holder} and an arbitrary integer $\mathtt n$ such that $\mathtt n \alpha > \mathtt m$.  Here, we want to choose $\alpha$ and $\mathtt n$ more carefully, so that $\mathtt n \leq \mathtt N$.  (This ensures that \eqref{z_1 upper bound} can be applied later.) In fact, $\mathtt n = \mathtt N$ will work, with a suitable $\alpha$.  To see this, we consider two cases:  $s \leq d-1$ and $s > d-1$.  Suppose we are in the first case.  Then the minimum in \eqref{Holder} is equal to $1/s$.  Since $s < \overline{s} \leq \mathtt N / \mathtt m$, there exists an $\alpha \in (0,1/s)$ such that $\mathtt N \alpha > \mathtt m$.  If we are in the second case, then the minimum in \eqref{Holder} is $\frac{d-s}{d-1}$, and $s < \overline{s} \leq d-\frac{(d-1)\mathtt m}{\mathtt N}$ implies that $\mathtt N \alpha > \mathtt m$ for some $\alpha \in (0,\frac{d-s}{d-1})$.  Having fixed $\alpha$ and $\mathtt n = \mathtt N$, the definition of $K$ proceeds just as in the proof of Proposition \ref{generalized-graph-lemma}.  This yields $\dim_{\mathrm H} K = s$ and utilizes, in particular, the function $F_s$ from Proposition \ref{high dimensional graph}.

Finally, to complete the proof, we assume that there exists some $\gamma \in (\Gamma \setminus \{0\}) \cap (K-K)$ and seek a contradiction.  This can be done as in the proof of Proposition \ref{generalized-graph-lemma} without modification, using in particular property \eqref{z bar lower bound}, the $\alpha$-H\"older continuity of $F_s$, and property \eqref{z_1 upper bound} with $n = \mathtt N$.
\end{proof}

\subsection{Anisotropic boxes and Hausdorff dimension}
\subsubsection{Proof of Lemma \ref{positivity}}
For $t \geq 0$, let $\mathscr{H}^t$ denote the standard $t$-dimensional Hausdorff measure on $\R^d$, defined by
\begin{align*}
\mathscr{H}^t(A) \colonequals \lim_{\delta \searrow 0}\mathscr{H}_\delta^t(A),
\end{align*}
where
\begin{align*}
\mathscr{H}_\delta^t(A) \colonequals \inf\Big\{\sum_{U \in \mathscr{U}} \diam(U)^t \colon \mathscr{U} \text{ is a countable cover of $A$, } \sup_{U \in \mathscr U} \diam U \leq \delta\Big\}.
\end{align*}
For any set $A \subset \R^d$, we have (by definition)
\begin{align}\label{Hausdorff dim def}
\dim_{\mathrm H} A = \inf\{t \colon \mathscr{H}^t(A) = 0\} = \sup\{t \colon \mathscr{H}^t(A) = \infty\};
\end{align}
see \cite[\S 2.2]{Mattila2}. Now, fix $E \subseteq \R^d$ and $s \in [0, \mathtt S - (d-\dim_{\mathrm H} E)\mathtt N)$, as in the lemma statement.  The definitions of $\mathtt S$ and $\mathtt N$ in \eqref{definition S, N} imply that $\mathtt S \leq d\mathtt N$.  It follows that
\begin{align*}
s = \mathtt S - (d-\alpha)\mathtt N
\end{align*}
for some $\alpha \in [0, \dim_{\mathrm H} E)$.  By \eqref{Hausdorff dim def}, we have $\mathscr{H}^\alpha(E) = \infty$, and thus there exists some $\delta \in (0,1]$ such that $\mathscr{H}_{\delta}^\alpha(E) > 0$.  Let $\fQ \subseteq \fD^*$ be an arbitrary cover of $E$.  If there exists some $Q_0 \in \fQ$ such that $\ell(Q_0)^{\mathtt N} \geq \delta/\sqrt{d}$, then
\begin{align*}
\sum_{Q \in \fQ}\ell(Q)^s \geq \ell(Q_0)^s \geq \Big(\frac{\delta}{\sqrt{d}}\Big)^{s/\mathtt N}.
\end{align*}
Suppose instead that every $Q \in \fQ$ obeys $\ell(Q)^{\mathtt N} \leq \delta/\sqrt{d}$.  Each $Q$  admits a covering $\mathscr{C}(Q)$ by exactly $\ell(Q)^{\mathtt S - d\mathtt N}$ \emph{cubes} of side-length $\ell(Q)^{\mathtt N}$.  In particular, the collection
\begin{align*}
\mathscr{C}(\fQ) \colonequals \bigcup_{Q \in \fQ}\mathscr{C}(Q)
\end{align*}
is a countable cover of $E$ consisting of sets of diameter at most $\delta$.  It follows that
\begin{align*}
\sum_{Q \in \fQ}\ell(Q)^s = \sum_{Q \in \fQ} \ell(Q)^{\mathtt S-d\mathtt N}\ell(Q)^{\alpha\mathtt N} = \sum_{C \in \mathscr{C}(\fQ)} \diam(C)^\alpha \geq \mathscr{H}_\delta^\alpha(E).
\end{align*}
Since $\fQ$ was arbitrary, we may conclude that
\begin{align*}
\fH_{\fD^*}^s(E) \geq \min\Big\{\Big(\frac{\delta}{\sqrt{d}}\Big)^{s/\mathtt N}, \mathscr{H}_\delta^\alpha(E)\Big\} > 0,
\end{align*}
completing the proof. \qed

\subsubsection{Proof of Lemma \ref{reversed inequality}}
Fix a set $E \subseteq \R^d$ such that $E$ is contained in an element of $\fD_{J}^*$ for some $J$.  Let $Q_E$ be the unique element of $\fD_J$ such that $E \subseteq Q_E$, and let $\fQ \subseteq \fD^*$ be an arbitrary cover of $E$.  On one hand, if $\fQ \not\subseteq \fD_J^*$, then there exists some $Q_0 \in \fQ$ such that $\ell(Q_0) \geq 2^{-J}$, and
\begin{align*}
\sum_{Q \in \fQ} \ell(Q)^s \geq \ell(Q_0)^s \geq 2^{-Js} = \ell(Q_E)^s \geq \fH_{\fD_J^*}^s(E).
\end{align*}
On the other hand, if $\fQ \subseteq \fD_{J}^*$, then
\begin{align*}
\sum_{Q \in \fQ} \ell(Q)^s \geq \fH_{\fD_J^*}^s(E)
\end{align*}
by definition.  Since $\fQ$ was arbitrary, we may conclude that $\fH_{\fD^*}^s(E) \geq \fH_{\fD_J^*}(E)$ and hence $\fH_{\fD^*}^s(E) = \fH_{\fD_J^*}(E)$ \qed

\subsubsection{Proof of Lemma \ref{alternate Frostman}}
Our argument follows the proof of the standard version of Frostman's lemma,  \cite[Theorem 8.8]{Mattila}, with minor adjustments. There, the idea is to create a sequence of measures $\{\vartheta^j\}_j$, with $\vartheta^j$ obeying the required ball condition on balls of diameter at least $2^{-j}$, and then take a weak limit. The main distinction here is that the dyadic cubes used in Frostman's proof  are replaced by the anisotropic boxes in $\mathcal D^*$.

Fix a compact set $E \subset \R^d$ and $s \geq 0$.  By translation, we may assume that $E$ is contained in some element of $\fD_{J}$ for some $J$.  For each $j \geq J$, define a measure $\vartheta_j^j$ on $\R^d$ by specifying that
\begin{align*}
\vartheta_j^j|_Q = \begin{cases} 0 &\text{if } E \cap Q = \emptyset\\
\frac{\ell(Q)^s}{\lambda(Q)}\lambda|_Q &\text{if } E \cap Q \neq \emptyset
\end{cases} \quad\quad \text{for each } Q \in \fD_j,
\end{align*}
where $\lambda$ denotes $d$-dimensional Lebesgue measure.  Suppose that measures $\vartheta_j^j, \vartheta_{j-1}^j, \ldots, \vartheta_{j-k}^j$ on $\R^d$ have been constructed, with $j - k > J$.  Define the measure $\vartheta_{j-k-1}^j$ by specifying that
\begin{align*}
\vartheta_{j-k-1}^j|_Q = \begin{cases} \vartheta_{j-k}^j|_Q &\text{if } \vartheta_{j-k}^j(Q) \leq \ell(Q)^s\\
\frac{\ell(Q)^s}{\vartheta_{j-k}^j(Q)}\vartheta_{j-k}^j|_Q &\text{if } \vartheta_{j-k}^j(Q) > \ell(Q)^s
\end{cases} \quad\quad \text{for each } Q \in \fD_{j-k-1}.
\end{align*}
Let
\begin{align*}
j^* \colonequals \max\{j' \leq j \colon E \subset Q \text{ for some } Q \in \fD_{j'}\},
\end{align*}
and define $\vartheta^j \colonequals \vartheta_{j^*}^j$ (noting that $J \leq j^* \leq j$). Let $Q_{j^*} \in \fD_{j^*}$ be such that $E \subset Q_{j^*}$.  The following are consequences of the construction of $\vartheta^j$:
\begin{enumerate}[1.]
\item{$\supp \vartheta^j \subseteq \bigcup\{Q \in \fD_j \colon E \cap Q \neq \emptyset\} \subseteq Q_{j^*}$;}
\item{$\vartheta^j(Q) \leq \ell(Q)^s$ for all $Q \in \fD_{j'}$ with $J \leq j' \leq j$;}
\item{Each point in $E$ belongs to some $Q \in \fD_{j'}$ with $J \leq j' \leq j$ such that $\vartheta^j(Q) = \ell(Q)^s$.}
\end{enumerate}
Statements 1 and 2 imply that
\begin{align}\label{total mass}
\sup_{j \geq J}\|\vartheta^j\| \leq \sup_{j \geq J}\ell(Q_{j^*})^s \leq 2^{-Js},
\end{align}
and consequently $\{\vartheta^j\}_{j \geq J}$ has a weakly convergent subsequence.  Let $\vartheta$ be the weak limit along this subsequence.

We will show that this measure satisfies the conclusions of the lemma.  Statement 1 and the hypothesis that $E$ is compact imply that $\vartheta$ is supported in $E$.  Fix $Q \in \fD^*$.  If $\ell(Q) \geq 2^{-J}$, then \eqref{total mass} implies that $\vartheta(Q) \leq \|\vartheta\| \leq 2^{-Js} \leq \ell(Q)^s$, while if $\ell(Q) \leq 2^{-J}$, then statement 2 implies that $\vartheta(Q) \leq \ell(Q)^s$.  This confirms the ``ball" condition for $\vartheta$.  It remains to show that the total mass of $\vartheta$ is at least $\fH_{\fD^*}^s(E)$.  Toward this end, fix $j \geq J$. Using statement 3, we select for each $x \in E$ the largest $Q \in \fD^*$ such that $x \in Q$ and $\vartheta^j(Q) = \ell(Q)^s$.  Let $\fQ_j$ denote the collection of these elements.  Then $\fQ_j$ covers both $E$ and $\supp \vartheta^j$, and distinct elements of $\fQ_j$ are disjoint.  It follows that
\begin{align*}
\|\vartheta^j\| = \sum_{Q \in \fQ_j}\vartheta^j(Q) = \sum_{Q \in \fQ_j}\ell(Q)^s \geq \fH_{\fD^*}^s(E).
\end{align*}
Since $j$ was arbitrary, we may conclude that $\|\vartheta\| \geq \fH_{\fD^*}^s(E)$. \qed

\subsubsection{Proof of Lemma \ref{energy bound}}
Fix $L$, $\sigma$, and $s$ as in the lemma statement, and let $\vartheta$ be any Borel measure $\vartheta$ supported on $[0,1]^d$ such that
\begin{align}\label{measure properties}
\|\vartheta\| \leq 1 \quad\quad \text{and} \quad\quad \sup_{Q \in \fD^*} \frac{\vartheta(Q)}{\ell(Q)^s} \leq L.
\end{align}
Let $\Omega \colonequals [0,2)^d$, and let $\Delta$ denote the diagonal of $\Omega \times \Omega$, namely $\Delta \colonequals \{(x,y) \in \Omega \times \Omega \colon x = y\}$.
Thus, the $\sigma$-dimensional energy of of $\vartheta$ can be written as
\begin{align*}
I_\sigma(\vartheta) = \iint_{\Omega \times \Omega \setminus \Delta} |x-y|^{-\sigma} d\vartheta(x)d\vartheta(y).
\end{align*}
We form a Whitney decomposition of $\Omega \times \Omega \setminus \Delta$ as follows:  For each $j \geq -1$, let $\mathcal{C}_j = \mathcal{C}_j(\Omega)$ be the set of dyadic cubes with side-length $2^{-j}$ contained in $\Omega$; that is,
\begin{align*}
\mathcal{C}_j \colonequals \{x + [0,2^{-j})^d \colon x \in 2^{-j}\Z^d \cap \Omega\}.
\end{align*}
We say that two dyadic cubes are \emph{adjacent} if their closures have nonempty intersection.  For $j \geq 0$, each cube in $\fC_j$ is contained in a unique ``parent" cube in $\fC_{j-1}$. We say that $C,C' \in \fC_j$ are \emph{related}, and write $C \sim C'$, if $C$ and $C'$ are nonadjacent but have adjacent parents.  The following properties are easy to confirm:
\begin{enumerate}[1.]
\item{For each $(x,y) \in \Omega \times \Omega \setminus \Delta$, there exists a (unique) pair of related cubes $C,C'$ such that $(x,y) \in C \times C'$.}
\item{If $C,C' \in \fC_j$ and $C \sim C'$, then $|x-y| \geq 2^{-j}$ for all $(x,y) \in C \times C'$.}
\item{For a fixed cube $C$, there are at most $6^d$ cubes $C'$ such that $C \sim C'$.}
\end{enumerate}
It follows from property 1 that
\begin{align*}
\Omega \times \Omega \setminus \Delta = \bigcup_{j = 0}^\infty \bigcup_{\substack{(C,C') \in \fC_j \times \fC_j \colon\\ C \sim C'}} C \times C'.
\end{align*}
Thus, using properties 2 and 3, together with \eqref{measure properties} and our hypothesis on $s$, we obtain
\begin{align*}
I_\sigma(\vartheta) &\leq \sum_{j=0}^\infty \sum_{\substack{(C,C') \in \fC_j \times \fC_j \colon\\ C \sim C'}} \iint_{C \times C'} |x-y|^{-\sigma} d\vartheta(x)d\vartheta(y)\\
&\leq \sum_{j=0}^\infty \sum_{C \in \fC_j} \sum_{C' \in \fC_j \colon C \sim C'} 2^{j\sigma}\vartheta(C)\vartheta(C')\\
&\leq \sum_{j=0}^\infty 2^{j\sigma} \sum_{C \in \fC_j}\vartheta(C) \#\{C' \in \fC_j \colon C \sim C'\} \max_{C' \in \fC_j}\vartheta(C')\\
&\leq 6^d\sum_{j=0}^\infty 2^{j\sigma}\Big(\sum_{C \in \fC_j}\vartheta(C)\Big)\max_{C' \in \fC_j}\#\{Q \in \fD_j \colon Q \cap C' \neq \emptyset\}L2^{-js}\\
&= 6^dL\sum_{j=0}^\infty 2^{j\sigma}\|\vartheta\| 2^{j(\mathtt S-d)}2^{-js} \leq 6^dL\sum_{j=0}^\infty 2^{j(\sigma+\mathtt S - d-s)} = \frac{6^dL}{1-2^{\sigma + \mathtt S - d - s}}.
\end{align*}
Since $\vartheta$ was arbitrary, we have shown that the lemma holds with
\begin{align*}
\mathtt E(t) \colonequals \frac{6^d}{1-2^{-t}},
\end{align*}
and the proof is complete.  \qed

\Addresses 
}

\end{document}